\documentclass[fleqn]{article}
\usepackage[utf8]{inputenc}
\usepackage{enumerate,amsthm,amssymb,xcolor,hyperref,comment}
\usepackage[margin=1in]{geometry}

\usepackage[leqno]{amsmath}

\makeatletter
\newcommand{\leqnomode}{\tagsleft@true}
\newcommand{\reqnomode}{\tagsleft@false}
\makeatother

\newtheorem{lemma}{Lemma}
\newtheorem{theorem}{Theorem}
\newcommand{\sset}[1]{\left\{#1\right\}}

\newcommand{\esspc}{(G, S, X_0, X, Y^*, f)}

\def\longbox#1{\parbox{0.85\textwidth}{#1}}

\title{Four-coloring $P_6$-free graphs.\\ I.  Extending an excellent precoloring}
\author{
Maria Chudnovsky\thanks{Supported by NSF grant DMS-1550991 and US Army Research Office grant W911NF-16-1-0404.}\\
Princeton University, Princeton, NJ 08544
\\
\\
Sophie Spirkl\\
Princeton University, Princeton, NJ 08544
\\
\\
Mingxian Zhong\\
Columbia University, New York, NY 10027}

\date{\today}
\begin{document}
\maketitle

\begin{abstract} 
  This is the first paper in a series whose goal is to give
  a polynomial time algorithm for the \textsc{$4$-coloring problem}
  and the \textsc{$4$-precoloring extension} problem restricted to the
  class of graphs with no induced six-vertex path, thus proving a conjecture 
of Huang. 
Combined with previously known 
results this completes the classification of the complexity of the $4$-coloring 
problem for graphs with a connected forbidden induced subgraph.

In this paper we
  give a polynomial time algorithm that determines if a special kind
  of precoloring of a $P_6$-free graph has a precoloring extension,
  and constructs such an extension if one exists. Combined with the
  main result of the second paper of the series, this gives a complete
  solution to the problem.
\end{abstract}

\section{Introduction} \label{sec:intro}

All graphs in this paper are finite and simple.
We use $[k]$ to denote the set $\sset{1, \dots, k}$. Let $G$ be a graph. A
\emph{$k$-coloring} of $G$ is a function
$f:V(G) \rightarrow [k]$. A $k$-coloring is {\em proper} if  for every edge 
$uv \in E(G)$,
$f(u) \neq f(v)$, and $G$ is \emph{$k$-colorable} if $G$ has a
proper $k$-coloring. The \textsc{$k$-coloring problem} is the problem of
deciding, given a graph $G$, if $G$ is $k$-colorable. This
problem is well-known to be $NP$-hard for all $k \geq 3$.

A function $L: V(G) \rightarrow 2^{[k]}$ that assigns a subset of
$[k]$ to each vertex of a graph $G$ is a \emph{$k$-list assignment}
for $G$. For a $k$-list assignment $L$, a function
$f: V(G) \rightarrow [k]$ is an \emph{$L$-coloring} if $f$ is a
$k$-coloring of $G$ and $f(v) \in L(v)$ for all $v \in V(G)$. A graph
$G$ is \emph{$L$-colorable} if $G$ has a proper  $L$-coloring. We denote by
$X^0(L)$ the set of all vertices $v$ of $G$ with $|L(v)|=1$.  The
\textsc{$k$-list coloring problem} is the problem of deciding, given a
graph $G$ and a $k$-list assignment $L$, if $G$ is 
$L$-colorable. Since this generalizes the $k$-coloring problem, it is
also $NP$-hard for all $k \geq 3$.

A \emph{$k$-precoloring} $(G, X, f)$ of a graph $G$ is a function
$f: X \rightarrow [k]$ for a set $X \subseteq V(G)$ such that $f$ is a
proper $k$-coloring of $G|X$. Equivalently, a $k$-precoloring is a
$k$-list assignment $L$ in which $|L(v)| \in \sset{1, k}$ for all
$v \in V(G)$. A \emph{$k$-precoloring extension} for $(G, X, f)$ is a
proper $k$-coloring $g$ of $G$ such that $g|_X = f|_X$, and the
\textsc{$k$-precoloring extension problem} is the problem of deciding,
given a graph $G$ and a $k$-precoloring $(G, X, f)$, if $(G, X, f)$
has a $k$-precoloring extension.

We denote by $P_t$ the path with $t$ vertices.  
Given a path $P$, its \emph{interior} is the set of vertices that
have degree two in $P$. We denote the interior of $P$ by $P^*$.
A  \emph{$P_t$ in a graph  $G$} is  a sequence $v_1-\ldots -v_t$ of pairwise 
distinct vertices where for $i,j \in [t]$, $v_i$ is adjacent to $v_j$ if and 
only if $|i-j|=1$.  We denote by $V(P)$ the set $\{v_1, \ldots, v_t\}$,
and if $a, b \in V(P)$, say $a=v_i$ and $b=v_j$ and $i<j$, then $a-P-b$ is the 
path $v_i-v_{i+1}-\ldots - v_j$.
A graph is \emph{$P_t$-free} if there is no $P_t$ in $G$.
Throughout the paper by ``polynomial time'' or ``polynomial size'' we mean 
running time,  or size, that is polynomial in $|V(G)|$.

Since the  \textsc{$k$-coloring problem} and the \textsc{$k$-precoloring extension problem} are $NP$-hard for $k \geq 3$, their
restrictions to graphs with a forbidden induced subgraph have been extensively 
studied; see  \cite{ICM,gps} for a survey of known results. In particular,
the following is known (given a graph $H$, we say that a graph $G$ is 
{\em $H$-free} if no  induced subgraph of $G$ is isomorphic to $H$):

\begin{theorem}[\cite{gps}] Let $H$ be a (fixed) graph, and let $k>2$. If
the \textsc{$k$-coloring problem} can be solved in polynomial time when restricted to the class of $H$-free graphs, then every connected component of $H$ is a path.
\end{theorem}

Thus if we assume that $H$ is connected, then the question of determining the 
complexity of $k$-coloring $H$-free graph is reduced to 
studying the complexity of coloring graphs with 
certain induced paths excluded,
and a significant body of work has been produced on this topic.
Below we list a few such results.

\begin{theorem}[\cite{c1}] \label{3colP7}
The \textsc{3-coloring problem} can be
  solved in polynomial time for the class of $P_7$-free graphs.
\end{theorem}

\begin{theorem}[\cite{hoang}] The \textsc{$k$-coloring problem} can be
  solved in polynomial time for the class of $P_5$-free graphs.
\end{theorem}

\begin{theorem}[\cite{huang}] The \textsc{4-coloring problem} is
  $NP$-complete for the class of $P_7$-free graphs.
\end{theorem}

\begin{theorem}[\cite{huang}] For all $k \geq 5$, the
  \textsc{$k$-coloring problem} is $NP$-complete for the class of
  $P_6$-free graphs.
\end{theorem}

The only cases for which the complexity of $k$-coloring $P_t$-free
graphs is not known are $k=4$, $t=6$, and $k=3$, $t \geq 8$.
This is the first paper in a series of two.
The main result of the series is the following:
\begin{theorem} 
\label{main}
The \textsc{4-precoloring extension problem} can be
  solved in polynomial time for the class of $P_6$-free graphs.
\end{theorem}
In contrast, the
\textsc{$4$-list coloring problem} restricted to $P_6$-free graphs is
$NP$-hard as proved by Golovach, Paulusma, and Song \cite{gps}.
As an immediate corollary of Theorem~\ref{main}, we obtain that 
the \textsc{$4$-coloring problem} for $P_6$-free graphs is also solvable in
polynomial time. This proves  a conjecture of Huang \cite{huang}, thus resolving the former open case above,  and completes the 
classification of the complexity of the \textsc{$4$-coloring problem} for graphs with a connected forbidden induced subgraph.

Let $G$ be a graph. For $X \subseteq V(G)$ we denote by
$G|X$ the subgraph induced by $G$ on $X$, and by
$G \setminus X$ the graph $G|(V(G) \setminus X)$.
If $X=\{x\}$, we write $G \setminus x$ to mean $G \setminus \{x\}$.
For disjoint subsets $A,B \subset V(G)$ we say that $A$ is \emph{complete} to $B$ if every vertex of $A$ is adjacent to every vertex of $B$, and that 
$A$ is \emph{anticomplete} to $B$ if every vertex of $A$ is non-adjacent to
every vertex of $B$. If $A=\{a\}$ we write $a$ is complete (or anticomplete)
to $B$ to mean $\{a\}$ that is complete (or anticomplete) to $B$.
If $a \not \in B$ is not complete and not anticomplete to $B$,
we say that $a$ is \emph{mixed} on $B$. Finally, if $H$ is an induced subgraph
of $G$ and $a \in V(G) \setminus V(H)$, we say that $a$ is \emph{complete to, 
anticomplete to}, or \emph{mixed on} $H$ if $a$ is complete to, anticomplete 
to,  or mixed on 
$V(H)$, respectively. For $v \in V(G)$ we write $N_G(v)$ (or $N(v)$ when there is no danger of confusion) to mean the set of vertices of $G$ that are adjacent to  $v$. Observe that since $G$ is simple, $v \not \in N(v)$. For $A \subseteq V(G)$,
an \emph{attachment} of $A$ is a vertex of $V(G) \setminus A$ complete to $A$.
For $B \subseteq V(G) \setminus A$ we denote by $B(A)$ the set of
attachments of $A$ in $B$.  If $F=G|A$, we sometimes
write $B(F)$ to mean $B(V(F))$.

Given a list assignment $L$ for $G$, we say that the pair $(G,L)$ is
colorable if $G$ is $L$-colorable.  For $X \subseteq V(G)$, we write
$(G|X,L)$ to mean the list coloring problem where we restrict the
domain of the list assignment $L$ to $X$.  Let $X \subset V(G)$ be
such that $|L(x)|=1$ for every $x \in X$, and let $Y \subset V(G)$. We
say that a list assignment $M$ is \emph{obtained from $L$ by updating
  $Y$ from $X$} if $M(v)=L(v)$ for every $v \not \in Y$, and
$M(v)=L(v) \setminus \bigcup_{x \in N(v) \cap X} \{L(x)\}$ for every
$v \in Y$.  If $Y=V(G)$, we say that $M$ is \emph{obtained from $L$ by
  updating from $X$}.  If $M$ is obtained from $L$ by updating from
$X^0(L)$, we say that $M$ is \emph{obtained from $L$ by updating}. Let
$L=L_0$, and for $i \geq 1$ let $L_i$ be obtained from $L_{i-1}$ by
updating. If $L_i=L_{i-1}$, we say that $L_i$ is \emph{obtained from
  $L$ by updating exhaustively}. Since
$0 \leq \sum_{v \in V(G)} |L_j(v)| < \sum_{v \in V(G)} |L_{j-1}(v)|
\leq 4|V(G)|$ for all $j < i$, it follows that $i \leq 4 |V(G)|$ and
thus $L_i$ can be computed from $L$ in polynomial time.

An \emph{excellent starred precoloring} of a graph $G$ is a six-tuple 
$P = (G, S, X_0, X, Y^*, f)$ such that 
\begin{enumerate}[(A)]
\item $f: S \cup X_0 \rightarrow \sset{1,2,3,4}$ is a proper coloring of $G|(S \cup X_0)$;
\item $V(G) = S \cup X_0 \cup X \cup Y^*$;
\item $G|S$ is connected and no vertex in $V(G) \setminus S$ is complete to $S$; \item every vertex in $X$ has neighbors of at least two different colors (with respect to $f$) in $S$;
\item no vertex in $X$ is mixed on a component of $G|Y^*$;  and
\item for every component of $G|Y^*$, there is a vertex in $S \cup X_0 \cup X$ complete to it. 
\end{enumerate}
We call $S$ the {\em seed} of $P$.
We define two list assignments associated with $P$. First, 
define $L_P(v) = \sset{f(v)}$ for every $v \in S \cup X_0$, and let
$L_P(v)=  \sset{1,2,3,4} \setminus (f(N(v) \cap S))$ for 
$v \not \in S \cup X_0$.  Second, $M_P$ is the list assignment obtained 
as follows.
First, define $M_1$ to be the list assignment for $G|(X \cup X_0)$
obtained from ${L_P}|\sset{X \cup X_0}$ by updating exhaustively;
let $X_1=\{x \in X \cup X_0 \; : \; |M_1(x_1)|=1\}$.
Now define $M_P(v)=L_P(v)$ if $v \not \in X \cup X_0$,
and $M_P(v)=M_1(v)$ if $v \in X \cup X_0$. 
Let $X^0(P)=X^0(M_P)$.  Then $S \cup X_0 \subseteq X^0(P)$.  A
\emph{precoloring extension of $P$} is a proper $4$-coloring $c$ of $G$ such
that $c(v)=f(v)$ for every $v \in S \cup X_0$; it follows that
$M_P(v)=\{c(v)\}$ for every $v \in X^0(P)$.  It will often be convenient
to assume that $X_0=X^0(P) \setminus S$, and this assumption can be made
without loss of generality. Note that in this case, $M_P(v) = L_P(v)$
for all $v \in X$.

For an excellent starred precoloring $P$ and a collection excellent
starred $\mathcal{L}$ of precolorings, we say that $\mathcal{L}$ is an
\emph{equivalent collection} for $P$ (or that $P$ is \emph{equivalent}
to $\mathcal{L}$) if $P$ has a precoloring extension if and only if at
least one of the precolorings in $\mathcal{L}$ has a precoloring
extension, and a precoloring extension of $P$ can be constructed from
a precoloring extension of a member of $\mathcal{L}$ in polynomial
time.

We break the proof of Theorem~\ref{main} into two independent parts, each handled in a separate paper of the series. 
In one part, we reduce the \textsc{4-precoloring extension problem}  
for  $P_6$-free graphs to determining if an excellent starred 
precolorings of a $P_6$-free graph has a precoloring extension, and finding one if it exists. In fact, we restrict the problem further, by ensuring that 
there is a universal bound (that works for all $4$-precolorings of all 
$P_6$-free graphs) on the size of the seed of the excellent starred precolorings that we need to consider. More precisely, we prove:

\begin{theorem}
\label{Yaxioms}
There exists an integer $C>0$ and  a polynomial-time algorithm with the following specifications.
\\
\\
{\bf Input:}  A 4-precoloring $(G,X_0,f)$ of a $P_6$-free graph $G$.
\\
\\
{\bf Output:}  A collection $\mathcal{L}$ of  excellent starred 
precolorings of $G$ such  that 
\begin{enumerate}
\item If for every $P' \in \mathcal{L}$ we can in polynomial time
  either find a precoloring extension of $P'$, or determine that none
  exists, then we can construct a 4-precoloring extension of
  $(G, X_0, f)$ in polynomial time, or determine that none exists:
\item $|\mathcal{L}| \leq |V(G)|^C$; and
\item for every $(G',S',X_0',X',Y^*,f') \in \mathcal{L}$, 
\begin{itemize}
\item $|S'| \leq C$;
\item $X_0 \subseteq S' \cup X_0'$;
\item $G'$ is an induced subgraph of $G$; and
\item $f'|X_0=f|X_0$.
\end{itemize}
\end{enumerate}
\end{theorem}

The proof of Theorem~\ref{Yaxioms}  is hard and technical, and we postpone it
to the second paper of the series \cite{Paper2}.
The other part of the proof of Theorem~\ref{main} is an algorithm that tests  
in  polynomial time if an excellent starred precoloring (where the size of the 
seed is fixed) has a precoloring extension.
The goal of the present paper is to solve this problem. We prove:

\begin{theorem}
\label{excellent}
For every positive integer $C$ there exists a polynomial-time algorithm with the following specifications.
\\
\\
{\bf Input:}  An excellent starred precoloring $P=\esspc$ of a 
$P_6$-free graph $G$ with $|S| \leq C$.
\\
\\
{\bf Output:} A precoloring extension of $P$ or a determination
that none exists.
\end{theorem}

Clearly, Theorem~\ref{Yaxioms} and Theorem~\ref{excellent} together imply
Theorem~\ref{main}.
The proof of Theorem~\ref{excellent} consists of several steps. At each step 
we replace the problem that 
we are trying to solve by a polynomially sized collection of simpler problems,
and the problems created in the last step can be encoded via 2-SAT.
Here is an outline of the proof. 
First we show that  an excellent starred precoloring $P$ of a $P_6$-free 
graph $G$  can be replaced by a polynomially sized collection $\mathcal{L}$ of 
excellent starred precolorings of $G$ that have an additional property 
(to which we refer  as   ``being orthogonal'') and  $P$ has a 
precoloring extension if and only  if some member of $\mathcal{L}$  does.
Thus in order 
to prove Theorem~\ref{excellent}, it is enough to be able to test if an 
orthogonal excellent starred precoloring of a $P_6$-free graph has a 
precoloring  extension. Our next step is an algorithm  whose 
input is an  orthogonal excellent starred precoloring $P$ of a $P_6$-free graph $G$, and whose output is a ``companion triple'' for $P$. A companion triple 
consists 
of a graph $H$ that may not be $P_6$-free, but certain parts  of it are, a
list assignment $L$ for $H$, and a correspondence function $h$ that establishes
the connection between $H$ and $P$. Moreover, in order to test if $P$ has
a precoloring extension, it is enough  to test if  $(H,L)$ is colorable.

The next  step of the algorithm is 
replacing $(H,L)$ by a polynomially sized collection  $\mathcal{M}$
of list assignments for $H$, such that $(H,L)$ is colorable if and only if
there exists $L' \in \mathcal{L}$ such that $(H,L')$ is colorable, and
in addition for every $L' \in \mathcal{L}$
the pair $(H,L')$ is ``insulated''.
Being insulated means that
$H$ is the union of four induced subgraphs $H_1, \ldots, H_4$, and 
in order to test if $(H,L')$ is colorable, it is 
enough to test if $(H_i, L')$ is colorable for each $i \in \{1,2,3,4\}$.
The final step of the algorithm is converting the problem of coloring
each $(H_i,L')$ into a $2$-SAT problem, and solving it in polynomial time.
Moreover, at each step of the proof, if a coloring exists, then  we can find 
it,  and convert in polynomial time into a precoloring extension of $P$.

This paper is organized as follows. In Section~\ref{sec:orthogonal} we
produce a collection $\mathcal{L}$ of orthogonal excellent starred 
precolorings.  In Section~\ref{sec:companion} we construct a companion triple 
for an orthogonal precoloring. In Section~\ref{insulatingcutsets} we 
start with a precoloring and its companion triple, and construct a
collection $\mathcal{M}$ of lists $L'$ such that every pair
$(H,L')$ is insulated. Finally, in Section~\ref{sec:2sat} we describe the 
reduction to 2-SAT. Section~\ref{sec:complete} contains the proof of 
Theorem~\ref{excellent} and of Theorem~\ref{main}.

\section{From Excellent to Orthogonal} \label{sec:orthogonal}

Let $P=\esspc$ be an excellent  starred   precoloring. 
For $v \in X \cup Y^*$, the \emph{type} of $v$ is the set $N(v) \cap S$.
Thus the number of possible types for a given precoloring is at most
$2^{|S|}$. In this section we will prove several lemmas that allow us to 
replace a given precoloring by an equivalent polynomially sized collection of
``nicer'' precolorings, with the additional property that the size of the seed
of each of the new precolorings is bounded by a function of the size of the
seed of the precoloring we started with. Keeping the size of the seed bounded 
allows us  to maintain the  property that the number of different types of 
vertices of $X \cup Y^*$ is bounded, and therefore, from the point of view of running 
time, we can always consider each type separately.

For $T \subseteq S$ we denote by $L_P(T)$ the set
$\{1,2,3,4\} \setminus \bigcup_{v \in T}\{f(v)\}$. Thus if $v$ is
of type $T$, then $L_P(v)=L_P(T)$. For $T \subseteq S$ and 
$U \subseteq X\cup Y^*$ we denote by
$U(T)$ the set of vertices of $U$ of type $T$.

A subset $Q$ of $X$ is \emph{orthogonal} if
there exist $a,b \in \{1,2,3,4\}$ such that for every $q \in Q$
either $M_P(q)=\{a,b\}$ or $M_P(q) = \{1,2,3,4\} \setminus \{a,b\}$.
We say that $P$ is \emph{orthogonal} if $N(y) \cap X$ is orthogonal for 
every $y \in Y^*$.

The goal of this section is to prove that for every excellent starred 
precoloring $P$ 
of a $P_6$-free graph $G$, there is a  an equivalent collection 
$\mathcal{L}(P)$ of  orthogonal   excellent starred precolorings of $G$. 
We start with a  few technical  lemmas.

\begin{lemma}
\label{Sophiespath}
Let $P=\esspc$ be an excellent  starred  precoloring of a $P_6$-free graph $G$.
Let  
$i,j \in \sset{1,2,3,4}$ and $k \in \{1,2,3,4\} \setminus \{i,j\}$.
Let $T_i,T_j$ be types such that $L_P(T_i)=\{i,k\}$
and $L_P(T_j)=\{j,k\}$, and let $x_i,x_i' \in X(T_i)$ and $x_j,x_j' \in X(T_j)$.
Suppose that $y_i,y_j \in Y^*$ are such that  $i,j \in M_P(y_i) \cap M_P(y_j)$, 
where 
possibly  $y_i=y_j$. Suppose further that the only possible edge among 
$x_i,x_i',x_j,x_j'$ is $x_ix_j$, 
and $y_i$ is adjacent to $x_i'$ and not to $x_i$, and $y_j$ is adjacent to 
$x_j'$ and not to $x_j$. Then there does not exist $y \in Y^*$
with $i,j \in M_P(y)$ and such that $y$ is complete to $\{x_i,x_j\}$ and
anticomplete to $\{x_i',x_j'\}$.
\end{lemma}

\begin{proof}
Suppose such $y$ exists. Since 
no vertex of $X$ is mixed on a component of $G|Y^*$, it follows that
$y$ is anticomplete to $\{y_i,y_j\}$. Since $x_i,x_i' \in X$ and
$i,k \in L_P(T_i)$, it follows that there exists $s_j \in T_i$ with
$L_P(s_j)=\{j\}$. 
Similarly, there exists $s_i \in Tj$ with $L_P(s_i)=\{i\}$. 
Since $i \in L_P(T_i)$ and $j \in L_P(T_j)$, it follows that
$s_i$ is anticomplete to $\{x_i,x_i'\}$ and $s_j$ is anticomplete to 
$\{x_j,x_j'\}$.

Since $i,j \in M_P(y_i) \cap M_P(y_j) \cap M_P(y)$ 
it follows that $\{s_i,s_j\}$  is anticomplete to $\{y_i,y_j,y\}$.
Since $x_i'-s_j-x_i-y-x_j-s_i-x_j'$ (possibly shortcutting through $x_ix_j$)
is not a $P_6$ in $G$, it follows that $s_i$ is adjacent to $s_j$.
If $y_i$ is non-adjacent to $x_j'$, and $y_j$ is non-adjacent to $x_i'$,
then $y_i \neq y_j$, and since $P$ is excellent, $y_i$ is non-adjacent to $y_j$, and so 
 $y_i-x_i'-s_j-s_i-x_j'-y_j$ is a $P_6$, 
a contradiction, so we may assume that $y_i$ is adjacent to $x_j'$.
 But now 
$x_j'-y_i-x_i'-s_j-x_i-y$
is a $P_6$, a contradiction. This proves Lemma~\ref{Sophiespath}.
\end{proof}

\begin{lemma}
\label{Sophiespath2}
Let $P=\esspc$ be an excellent  starred  precoloring  of a $P_6$-free graph $G$.
Let $\{i,j,k,l\}=\{1,2,3,4\}$. 
Let $T_i,T_j$ be types such that $L_P(T_i)=\{i,k\}$
and $L_P(T_j)=\{j,k\}$, and let $x_i,x_i' \in X(T_i)$ and $x_j,x_j' \in X(T_j)$.
Let $y_i^i,y_j^i \in Y^*$ with $i,l \in M_P(y_i^i) \cap M_P(y_j^i)$,
and let $y_i^j,y_j^j \in Y^*$ with $j,l  \in M_P(y_i^j) \cap M_P(y_j^j)$,
where possibly $y_i^i=y_j^i$ and $y_i^j=y_j^j$.
Assume that
\begin{itemize}
\item  some component $C_i$ of $G|Y^*$ contains both $y_i^i,y_i^j$; 
\item some component $C_j$  of $G|Y^*$ contains both $y_j^i,y_j^j$;
\item for every $t \in \{i,j\}$ there is a path  $M$ in $C_t$ from
$y_t^i$ to $y_t^j$  with $l \in M_P(u)$ for every $u \in V(M)$;
\item the only possible edge among 
$x_i,x_i',x_j,x_j'$ is $x_ix_j$;
\item $y_i^i,y_i^j$ are adjacent to $x_i'$ and not to $x_i$; 
\item $y_j^i,y_j^i$ are adjacent to  $x_j'$ and not to $x_j$.
\end{itemize}
Then there do not exist  $y^i,y^j \in Y^*$
with $i,l \in M_P(y^i)$, $j,l \in M_P(y^j)$ and  such that 
\begin{itemize}
\item some component $C$ of $G|Y^*$ 
contains both $y^i$ and $y^j$, and
\item $l \in M_P(u)$ for every $u \in V(C)$, and
\item  $\{y^i,y^j\}$ is complete to  $\{x_i,x_j\}$ and anticomplete to 
$\{x_i',x_j'\}$.
\end{itemize}
\end{lemma}

\begin{proof}
Suppose such $y^i,y^j$ exist. 
Since $P$ is an excellent starred precoloring,
no vertex of $X$ is mixed on a component  of $G|Y^*$, and therefore
$V(C)$ is anticomplete to $V(C_i) \cup V(C_j)$.
Since $x_i,x_i' \in X$ and
$i,k \in L_P(T_i)$, it follows that there exists $s_j \in T_i$ with
$L_P(s_j)=\{j\}$. 
Similarly, there exists $s_i \in Tj$ with $L_P(s_i)=\{i\}$. 
Since $i \in L_P(T_i)$ and $j \in L_P(T_j)$, it follows that
$s_i$ is anticomplete to $\{x_i,x_i'\}$ and $s_j$ is anticomplete to 
$\{x_j,x_j'\}$. Since $i \in M_P(y^i) \cap M_P(y_i^i) \cap M_P(y_j^i)$, it follows that $s_i$ is anticomplete to
$\{y^i,y_i^i,y_j^i\}$, and similarly $s_j$ is anticomplete to 
$\{y^j,y_i^j,y_j^j\}$.

First we prove that $s_i$ is adjacent to $s_j$. Suppose not.
Since $x_i'-s_j-x_i-x_j-s_i-x_j'$ is not a $P_6$ in $G$, it
follows that $x_i$ is non-adjacent to $x_j$. 
But now $x_i'-s_j-x_i-y^j-x_j-s_i$ or  $x_i'-s_j-x_i-y^j-s_j-x_j'$ is a $P_6$ 
in $G$, a contradiction. This proves that $s_i$ is adjacent to $s_j$.

If $y_i^j$ is adjacent to $x_j'$, then   $x_j'-y_i^j-x_i'-s_j-x_i-y^j$ is a $P_6$,
a contradiction. Therefore  $x_j'$ is non-adjacent to $y_i^j$, and therefore
$x_j'$ is anticomplete to $C_i$. Similarly, $x_i'$ is anticomplete
to $C_j$. In particular it follows that $C_i \neq C_j$.

Since $L_P(T_j)=\{i,k\}$ there exists $s_l \in S$ with $L_P(s_l)=\sset{l}$ such that
$s_l$ is complete to $X(T_j)$. Since $l \in M_P(y)$ for every
$y \in \{y_i^i,y_i^j,y_j^i,y_j^j,y^i,y^j\}$, it follows
that $s_l$ is anticomplete to $\{y_i^i,y_i^j,y_j^i,y_j^j,y^i,y^j\}$.
Recall that $x_i,x_i' \in X(T_i)$, and so no vertex of $S$ is mixed
on $\{x_i,x_i'\}$. Similarly no vertex of $S$ is mixed on $\{x_j,x_j'\}$.
If $s_l$ is anticomplete to $\{x_i,x_i'\}$, then one of
$y_i^j-x_i'-s_j-s_l-x_j'-y_j^j$,
$x_i'-s_j-x_i-y^j-x_j-s_l$,   $x_i'-s_j-x_i-x_j-s_l-x_j'$  
is a $P_6$, so $s_l$ is complete  to $\{x_i,x_i'\}$.

Since $y_i^i-x_i'-s_j-s_i-x_j'-y_j^j$ is not
a $P_6$, it follows that either $s_j$ is adjacent to $y_i^i$, or
$s_i$ is adjacent to $y_j^j$. We may assume that $s_j$ is adjacent
to $y_i^i$.

Let $M$ be a path in $C_i$ from $y_i^j$ to $y_i^i$ with 
$l \in M_P(u)$ for every $u \in V(M)$.
Since $s_j$ is adjacent to $y_i^i$ and not to $y_i^j$, there is
exist adjacent $a,b \in V(M)$ such that $s_j$ is adjacent to $a$ and not to 
$b$. Since $l \in M_P(u)$ for every $u \in V(M)$, it follows that
$s_l$ is anticomplete to $\{a,b\}$.
But now if $s_l$ is non-adjacent to $s_j$, then $b-a-s_j-x_i-s_l-x_j'$ 
is a $P_6$, and if $s_l$ is adjacent to $s_j$, then
$b-a-s_j-s_l-x_j'-y_j^j$  is a $P_6$; in both cases a contradiction.
This proves  Lemma~\ref{Sophiespath2}. 
\end{proof}

Let $P=\esspc$ be an excellent starred precoloring of a $P_6$-free graph $G$.
Let $S'' \subseteq  X$, and let $X_0'' \subseteq X \cup Y^*$. 
Let $f': S \cup X_0 \cup S'' \cup X_0''  \rightarrow \{1,2,3,4\}$ be 
such that $f'|(S \cup X_0)=f|(S \cup X_0)$ and 
$(G, S \cup X_0 \cup S'' \cup X_0'', f')$
is a 4-precoloring of $G$. 
Let  $X''$ be the set of vertices $x$ of $X \setminus X_0''$ such that
$x$ as a neighbor $z \in S''$ with $f'(z) \in M_P(x)$.
Let 
$$S'=S \cup  S''$$
$$X_0'=X_0 \cup X'' \cup X_0''$$
$$X'=X \setminus (X'' \cup S'' \cup X_0'')$$ 
$${Y^*}'=Y^* \setminus X_0''.$$

We say that $P'=(G,S',X_0',X',{Y^*}',f')$ is  \emph{obtained from $P$ by moving 
$S''$ to the seed  with colors $f'(S'')$, and moving $X_0''$ to $X_0$ with 
colors  $f'(X_0'')$}. Sometimes we say that ``we 
move $S''$ to $S$ with colors $f'(S'')$, and  $X_0''$ to $X_0$ with colors 
$f'(X_0'')$''.

In the next lemma we show that this operation creates another excellent starred
precoloring.

\begin{lemma} 
\label{movetoseed}
Let $P=\esspc$ be an excellent starred precoloring of a $P_6$-free graph $G$.
Let $S'' \subseteq X$ and $X_0'' \subseteq X \cup Y^*$, and let 
$S',X_0',X', {Y^*}',f'$ be as above. 
Then either $P'=(G,S',X_0',X',{Y^*}',f')$ is an excellent starred 
precoloring.
\end{lemma}

\begin{proof}
We need to check the following conditions:
\begin{enumerate}
\item $f': S' \cup X_0' \rightarrow \sset{1,2,3,4}$ is a proper coloring of $G|(S' \cup X_0')$;
\item $V(G) = S' \cup X_0' \cup X' \cup {Y^*}'$;
\item $G|S'$ is connected and no vertex in $V(G) \setminus S'$ is complete to 
$S'$; 
\item every vertex in $X'$ has neighbors of at least two different colors (with respect to $f'$) in $S'$;
\item no vertex in $X'$ is mixed on a component of $G|{Y^*}'$;  and
\item for every component of $G|{Y^*}'$, there is a vertex in $S' \cup X_0' \cup X'$ complete to it.
\end{enumerate}

Next we check the conditions.
\begin{enumerate}
\item  holds by the definition of $P'$. 
\item holds since  $S' \cup X_0' \cup X' \cup {Y^*}'=S \cup X_0 \cup X \cup Y^*$.
\item $G|S'$ is connected since $G|S$ is connected, and every $z \in S''$ has
a neighbor in  $S$. Moreover, since no vertex of $V(G) \setminus S$ is 
complete to $S$,
it follows that no vertex of $V(G) \setminus S'$ is complete to $S'$.
\item follows from the fact that $X' \subseteq X$.
\item follows from the fact that ${Y^*}' \subseteq Y^*$ and $X' \subseteq X$.
\item follows from the fact that  ${Y^*}' \subseteq Y^*$ and $S \cup X_0 \subseteq S' \cup X_0'$.
\end{enumerate}
\end{proof}

Let $P=\esspc$ be an excellent starred precoloring.
Let $i,j \in \{1,2,3,4\}$. Write $X_{ij}=\{x \in X \text{ such that }M_P(x)=\{i,j\}\}.$
For $y \in Y^*$ let $C_P(y)$ (or $C(y)$ when there is no danger of confusion) denote the vertex set of the component
of $G|Y^*$ that contains $y$. 

Let $P=\esspc$ be an excellent starred  precoloring, and let
$\{i,j,k,l\}= \{1,2,3,4\}$. We  say that $P$
is \emph{$kl$-clean} if there does not exist
$y \in Y^*$ with the following properties:
\begin{itemize}
\item  $i,j \in M_P(y)$, and 
\item there is $u \in C(y)$ with $k \in M_P(u)$, and 
\item  $y$ has both a neighbor in $X_{ik}$ and a neighbor in $X_{jk}$.
\end{itemize}
We say that
$P$ is \emph{clean} if it is $kl$-clean for every $k,l \in \{1,2,3,4\}$.

We say that $P$  is \emph{$kl$-tidy} if
there do not exist vertices $y_i,y_j \in Y^*$
such that 
\begin{itemize}
\item $i \in M_P(y_i)$, $j \in M_P(y_j)$, and 
\item $C(y_i)=C(y_j)$, and
\item there is a path $M$ from $y_i$ to $y_j$ in $C$ such that $l \in M_P(u)$
for every $u \in V(M)$, and
\item there is $u \in V(C)$ with  $k \in M_P(u)$, and
\item $y_i$ has a neighbor in $X_{ki}$ and a neighbor in $X_{kj}$
\end{itemize}
Observe that since no vertex of $X$ is mixed on an a component of $G|Y^*$,
it follows that $N(y_i) \cap X_{ki}$ is precisely the set of vertices of
$X_{ki}$ that are complete to $C(y_i)$, and an analogous statement holds for
$X_{kj}$.
We say that $P$ is  \emph{tidy} if it 
is $kl$-tidy for every $k,l \in \{1,2,3,4\}$.

We say that $P$ is \emph{$kl$-orderly} if  
for every  $y$ in  $Y^*$ with $\{i,j\} \subseteq M_P(y)$, 
$N(y) \cap X_{ik}$ is complete to  $N(y) \cap  X_{jk}$.
We say that $P$ is \emph{orderly} if it is $kl$-orderly  for every 
$k,l \in \{1,2,3,4\}$

Finally, we say that $P$ is \emph{$kl$-spotless} if no vertex  $y$ in  
$Y^*$ with $\{i,j\} \subseteq M_P(y)$ has
both a neighbor in $X_{ik}$ and a neighbor in $X_{jk}$.
We say that $P$ is \emph{spotless} if it is $kl$-spotless for every 
$k,l \in \{1,2,3,4\}$

Our goal is to replace an excellent starred precoloring by an 
equivalent  collection of spotless  precolorings. 
First we prove a lemma that allows us to replace an excellent starred 
precoloring with an equivalent collection of clean precolorings.

\begin{lemma}
  \label{clean} There is a function
  $q : \mathbb{N} \rightarrow \mathbb{N}$ such that the following
  holds.  Let $G$ be a $P_6$-free graph, and let $P=\esspc$ be an
  excellent starred precoloring of $G$. Then there is an 
algorithm
  with running time $O(|V(G)|^{q(|S|)})$ that outputs a collection
  $\mathcal{L}$ of excellent starred precolorings of $G$ such that:
\begin{itemize}
\item $|\mathcal{L}| \leq |V(G)|^{q(|S|)}$;
\item $|S'| \leq q(|S|)$ for every $P' \in \mathcal{L}$; 
\item every $P' \in \mathcal{L}$ is $kl$-clean for every $(k,l)$ for which
$P$ is $kl$-clean;
\item every $P' \in \mathcal{L}$ is $14$-clean;
\item $\mathcal{L}$ is an equivalent collection for $P$.
\end{itemize}
\end{lemma}

\begin{proof}
Without loss of generality we may assume that $X_0=X^0(P) \setminus S$.
Thus $L_P(x)=M_P(x)$ for every $x \in X$. 
We may assume that $P$ is not $14$-clean for otherwise we may set 
$\mathcal{L}=\{P\}$.
Let $Y$ be the set of vertices of $Y^*$ with $2,3 \in M_P(y)$
and such that some $u \in C(y)$ has $1 \in M_P(u)$.
Let $T_1, \ldots, T_p$ be the subsets of $S$ with $L_P(T_s)=\{1,2\}$ and
$T_{p+1}, \ldots, T_m$ the subsets of $S$ with $L_P(T_s)=\{1,3\}$.
Let $\mathcal{Q}$ be the collection of all $m$-tuples 
$$((S_1,Q_1), (S_2,Q_2), \ldots, (S_m,Q_m))$$
where for every $r \in \{1, \ldots, m\}$
\begin{itemize}
\item  $S_r \subseteq X(T_r)$ and $|S_r| \in \{0,1\}$, 
\item if $S_r =\emptyset$, then $Q_r=\emptyset$
\item $S_r=\{x_r\}$ then $Q_r=\{y\}$ where $y \in Y \cap N(x_r)$.
\end{itemize}

For $Q \in \mathcal{Q}$ construct a precoloring $P_Q$ as follows.
Let $r \in \{1, \ldots, m\}$. We may assume that $r \leq p$.
\begin{itemize}
\item Assume first   that $S_r=\{x_r\}$. Then $Q_r=\{y_r\}$. 
Move $\{x_r\}$ to the seed with color $1$, and for every $y \in Y$ such 
that $N(y) \cap X(T_r) \subset N(y_r) \cap X(T_r) \setminus \{x_r\}$,
move $N(y) \cap X(T_r)$ to $X_0$ with the unique color of 
$L_P(T_r) \setminus \{1\}$.
\item Next assume that $S_r=\emptyset$. Now for every $y \in Y$
move $N(y) \cap X(T_r)$ to $X_0$ with the unique color of 
$L_P(T_r) \setminus \{1\}$.
\end{itemize}
In the notation of Lemma~\ref{movetoseed}, if the precoloring of 
$G|(X_0' \cup S')$ thus obtained is not proper, 
remove $Q$ form $\mathcal{Q}$. Therefore we may assume that the precoloring is 
proper.
Repeatedly applying Lemma~\ref{movetoseed} we deduce that $P_Q$ is an 
excellent starred precoloring. Observe that ${Y^*}'=Y^*$.
Since $X' \subseteq X$ and ${Y^*}'=Y^*$, it follows that if $P$ is $kl$-clean, 
then so is $P_Q$. 

Now we show that $P_Q$ is $14$-clean. 
Let $Y'$ be the set of vertices $y$ of $Y^*$ such that
$2,3 \in M_{P_Q}(y)$ and some vertex $u \in C(y)$ has $1 \in M_{P_Q}(u)$.
Observe that $Y' \subseteq Y$.
It is enough to check that no vertex of $Y'$
has both a neighbor in $X'_{12}$ and a neighbor in $X'_{13}$. 
Suppose this is false, and suppose that  $y \in Y'$
has a neighbor $x_2 \in X'_{12}$ and a neighbor  
$x_3 \in X'_{13}$. 
Then  $x_2 \in X_{12}$ and $x_3 \in X_{13}$. We may assume that
$x_2 \in X(T_1)$ and $x_3 \in X(T_{p+1})$. Since $x_2,x_3 \not \in X^0(P_Q)$,
it follows that both $S_1 \neq \emptyset$ and $S_{p+1} \neq \emptyset$,
and therefore $Q_1 \neq \emptyset$ and $Q_{p+1} \neq \emptyset$.
Write $S_1=\{x_2'\}$, $Q_1=\{y_2\}$, $S_{p+1}=\{x_3'\}$ and $Q_{p+1}=\{y_3\}$. 
Since some $u \in C(y)$ has $1 \in M_{P_Q}(u)$, and since $x_2',x_3'$ are not mixed on $C(y)$, it follows that $y$ is anticomplete 
to  $\{x_2',x_3'\}$. Again since $x_2 \not \in X^0(P_Q)$, it follows that
$N(y) \cap X(T_1) \not \subseteq N(y_2) \cap X(T_1)$, and so we may assume that 
$x_2 \not \in N(y_2)$. Similarly, we may assume that $x_3 \not \in N(y_3)$.
But now the vertices $x_2,x_2',x_3,x_3',y_2,y_3,y$ contradict Lemma~\ref{Sophiespath}. This proves that $P_Q$ is $14$-clean.

Since $S'=S \cup \bigcup_{i=1}^mS_i$,
and since $m \leq 2^{|S|}$,  it follows that $|S'| \leq |S|+m \leq |S|+2^{|S|}$.

Let $\mathcal{L}=\{P_Q \; : \; Q \in \mathcal{Q}\}$.
Then  $|\mathcal{L}| \leq |V(G)|^{2m} \leq |V(G)|^{2^{|S|+1}}$.
We show that $\mathcal{L}$ is an equivalent collection for $P$. Since every 
$P' \in \mathcal{L}$ is obtained from $P$ by precoloring some vertices and 
updating,
it is clear that if $c$ is a precoloring extension of 
a member of $\mathcal{L}$, then $c$ is a precoloring extension of $P$. 
To see the converse, let $c$ be a precoloring 
extension of $P$. For every
$i \in \{1, \ldots, m\}$ define $S_i$ and $Q_i$ as follows. If no vertex
of $Y$ has a neighbor $x \in X(T_i)$ with $c(x)=1$, set 
$S_i=Q_i=\emptyset$. If some vertex of $Y$ has neighbor $x \in X(T_i)$
with $c(x)=1$, let $y$ be a vertex with this property and in addition with 
 $N(y) \cap X(T_i)$ minimal; let $x \in X(T_i) \cap N(y)$ with $c(x)=1$; 
and set $Q_i=\{y\}$ and $S_i=\{x\}$. Let $Q=((S_1,Q_1), \ldots, (S_m,Q_m))$.
We claim that $c$ is a precoloring extension of $P_Q$. Write
$P_Q=(G,S',X_0',X', Y',f')$. We need to show that
$c(v)=f'(v)$ for every $v \in S' \cup X_0'$. Since $c$ is a precoloring extension of $P$, it follows that $c(v)=f(v)=f'(v)$ for every $v \in S \cup X_0$.
Since $S' \setminus S = \bigcup_{s=1}^mS_s$ and $c(v)=f'(v)=1$ for
every $v \in  \bigcup_{s=1}^mS_s$, we deduce that $c(v)=f'(v)$ for every 
$v \in S'$. Finally let $v \in X_0' \setminus X_0$.
It follows that $v \in X$,
$f'(v)$ is the unique color of $M_P(v) \setminus \{1\}$, and there
are three possibilities.
\begin{enumerate}
\item $1 \in M_P(v)$ and $v$ has a neighbor in $\bigcup_{s=1}^mS_s$, or
\item there is $i \in \{1, \ldots, m\}$ with $S_i=\{x_i\}$ and $Q_i=\{y_i\}$,
and there is $y \in Y^*$ such that 
$N(y) \cap X(T_i) \subseteq (N(y_i) \cap X(T_i)) \setminus \{x_i\}$, and
$v \in N(y) \cap X(T_i)$, or 
\item there is $i \in \{1, \ldots, m\}$ with $S_i=Q_i=\emptyset$,
and there is $y \in Y^*$ such that  $v \in N(y) \cap X(T_i)$.
\end{enumerate}
We show that in all these cases $c(v)=f'(v)$.
\begin{enumerate}
\item Let $x \in  \bigcup_{s=1}^mS_s$. Then $c(x)=1$, and so 
$c(v) \neq 1$, and thus $c(v)=f'(v)$.
\item By the choice of $y_i$ and since 
$N(y) \cap X(T_i) \subseteq (N(y_i) \cap X(T_i)) \setminus \{x_i\}$, it follows
that $c(u) \neq 1$ for every  $u \in N(y) \cap X(T_i)$, and therefore
$c(v)=f'(v)$.
\item Since $S_i=\emptyset$, it follows that for every  $y' \in Y^*$ and 
for every $u \in N(y') \cap X(T_i)$ we have that $c(u) \neq 1$, and again
$c(v)=f'(v)$.
\end{enumerate}
This proves that $c$ is a precoloring extension of $P_Q$, and completes the 
proof of  Lemma~\ref{clean}.
\end{proof}

Repeatedly applying Lemma~\ref{clean} and using symmetry, we deduce the 
following:

\begin{lemma}
  \label{clean2}  There is a function
  $q : \mathbb{N} \rightarrow \mathbb{N}$ such that the following
  holds.  Let $G$ be a $P_6$-free graph.  Let $P=\esspc$ be an
  excellent starred precoloring of $G$. Then there is an 
algorithm
  with running time $O(|V(G)|^{q(|S|)})$ that outputs a collection
  $\mathcal{L}$ of excellent starred precolorings of $G$ such that:
\begin{itemize}
\item $|\mathcal{L}| \leq |V(G)|^{q(|S|)}$;
\item $|S'| \leq q(|S|)$ for every $P' \in \mathcal{L}$;
\item every $P' \in \mathcal{L}$ is clean;
\item $\mathcal{L}$ is an equivalent collection for $P$.
\end{itemize}
\end{lemma}

Next we show that a clean precoloring can be replaced with an equivalent 
collection of precolorings that are both clean and tidy.

\begin{lemma}
\label{tidy} 
 There is a function $q : \mathbb{N} \rightarrow \mathbb{N}$ such that
the following holds.  Let $G$ be a $P_6$-free graph.  Let $P=\esspc$
be a clean excellent starred precoloring of $G$. Then there 
is an
algorithm with running time $O(|V(G)|^{q(|S|)})$ that outputs a
collection $\mathcal{L}$ of excellent starred precolorings of $G$ such
that:
\begin{itemize}
\item $|\mathcal{L}| \leq |V(G)|^{q(|S|)}$;
\item $|S'| \leq q(|S|)$ for every $P' \in \mathcal{L}$; 
\item every $P' \in \mathcal{L}$ is clean;
\item every $P' \in \mathcal{L}$ is $kl$-tidy for every $k,l$ for which
$P$ is $kl$-tidy;
\item every $P' \in \mathcal{L}$ is $14$-tidy;
\item $\mathcal{L}$ is an equivalent collection for $P$.
\end{itemize}
\end{lemma}

\begin{proof}
Without loss of generality we may assume that $X_0=X^0(P) \setminus S$,
and thus $L_P(x)=M_P(x)$ for every $x \in X$.
We may assume that $P$ is not $14$-tidy for otherwise we may set 
$\mathcal{L}=\{P\}$.  Let $Y$ be the set of all pairs 
 $(y_2,y_3)$ with $y_2, y_3 \in Y^*$
such that 
\begin{itemize}
\item $2 \in M_P(y_2)$, $3 \in M_P(y_3)$, 
\item $y_2,y_3$ are in the same component $C$ of $G|Y^*$,
\item there is a path $M$ from $y_2$ to $y_3$ in $C$ such that $4 \in M_P(u)$
for every $u \in V(M)$, and
\item for some $u \in V(C)$, $1 \in M_P(u)$,
\end{itemize}

Let $T_1, \ldots, T_p$ be the subsets of $S$  with $L_P(T_s)=\{1,2\}$ and let
$T_{p+1}, \ldots, T_m$ be the subsets of  $S$ with $L_P(T_s)=\{1,3\}$.
Let $\mathcal{Q}$ be the collection of all $m$-tuples 
$$((S_1,Q_1), (S_2,Q_2), \ldots, (S_m,Q_m))$$
where for $r \in \{1, \ldots, m\}$
\begin{itemize}
\item  $S_r \subseteq X(T_r)$ and $|S_r| \in \{0,1\}$, 
\item if $S_r =\emptyset$, then $Q_r=\emptyset$
\item $S_r=\{x_r\}$ then $Q_r=\{(y_2^r,y_3^r)\}$ where $(y_2^r,y_3^r) \in Y$
and $x_r$ is complete to $\{y_2^r,y_3^r\}$.
\end{itemize}

For $Q \in \mathcal{Q}$ construct a precoloring
$P_Q=(G^Q,S^Q,X_0^Q,X^Q,Y^Q, f^Q)$ as follows.  Let
$r \in \{1, \ldots, m\}$; for $r = 1, \dots, m$, we proceed as
follows.
\begin{itemize}
\item Assume first   that $S_r=\{x_r\}$. Then $Q_r=\{(y_2^r, y_3^r)\}$. 
Move $x_r$ to the seed with color $1$, and for every $(y_2,y_3) \in Y$ such 
that $N(y_2) \cap X(T_r) \subset N(y_2^r) \cap (X(T_r) \setminus \{x_r\})$,
move $N(y_2) \cap X(T_r)$ to $X_0$ with the unique color of $L_P(T_r) \setminus \{1\}$.
\item Next assume that $S_r=\emptyset$. Now for every $y \in Y$
move $N(y) \cap X(T_r)$ to $X_0$ with the unique color of 
$L_P(T_r) \setminus \{1\}$.
\end{itemize}
In the notation of Lemma~\ref{movetoseed}, if the precoloring of 
$G|(X_0' \cup S')$ thus obtained is not proper, 
remove $Q$ form $\mathcal{Q}$. Therefore we may assume that the precoloring is 
proper.
Repeatedly applying Lemma~\ref{movetoseed} we deduce that 
$P_Q$ is an excellent starred precoloring. 
Observe that $Y^Q=Y^*$, $M_{P_Q}(y) \subseteq M_P(y)$ for every $y \in Y^Q$,
and $M_{P_Q}(x)=M_P(x)$ for every $x \in X^Q \setminus X^0(P_Q)$.
It follows that $P_Q$ is clean, and that if $P$ is 
$kl$-tidy, then so is $P_Q$.

Now we show that $P_Q$ is $14$-tidy. Suppose that 
there exist  $y_2,y_3 \in Y^Q$ that violate the definition 
of being $14$-tidy. Let
$x_2 \in X^Q_{12}$ and  $x_3 \in X^Q_{13}$ be adjacent to $y_2$, say, and 
therefore complete to $\{y_2,y_3\}$.
We may assume that $x_2 \in X(T_1)$ and $x_3 \in X(T_{p+1})$. 
Since $x_2,x_3 \not \in X^0(P_Q)$,
it follows that both $S_1 \neq \emptyset$ and $S_{p+1} \neq \emptyset$,
and therefore $Q_1 \neq \emptyset$ and $Q_{p+1} \neq \emptyset$.
Write $S_1=\{x_2'\}$, $Q_1=\{(y_2^2,y_3^2)\}$, $S_{p+1}=\{x_3'\}$ and 
$Q_{p+1}=\{y_2^3,y_3^3\}$. 

Since there is a vertex $u$ in the component of $G|Y^Q$ containing  
$y_2,y_3$ with $1 \in M_{P_Q}(u)$, and since no vertex of $X$ is mixed on
a component of $Y^*$,
it follows that 
$\{y_2,y_3\}$ is  anticomplete to 
$\{x_2',x_3'\}$. Since $x_2 \not \in X^0(P_Q)$, it follows that
$N(y_2) \cap X(T_1) \not \subseteq N(y_2^2) \cap (X(T_1) \setminus \{x_2'\})$, 
and so we may assume 
that  $x_2 \not \in N(y_2^2)$. Similarly, we may assume that 
$x_3 \not \in N(y_2^3)$.
But now, since no vertex of $X$ is mixed on a component of $Y^*$,
we deduce that  the vertices $x_2,x_2',x_3,x_3',y_2^2,y_2^3,y_3^2,y_3^3,y_2,y_3$ 
contradict Lemma~\ref{Sophiespath2}. This proves that $P_Q$ is $14$-tidy.

Since $S'=S \cup \bigcup_{i=1}^mS_i$,
and since $m \leq 2^{|S|}$, 
it follows that $|S'| \leq |S|+m \leq |S|+2^{|S|}$.

Let $\mathcal{L}=\{P_Q \; : \; Q \in \mathcal{Q}\}$.
Then  $|\mathcal{L}| \leq |V(G)|^{3m} \leq |V(G)|^{3 \times 2^{|S|}}$.
We show that $\mathcal{L}$ is an equivalent collection for $P$. Since every 
$P' \in \mathcal{L}$ is obtained from $P$ by precoloring some vertices and 
updating,
it is clear that every precoloring extension of a member of $\mathcal{L}$ 
is a precoloring extension of $P$.
To see the converse, suppose that $P$ has a precoloring 
extension $c$. For every
$i \in \{1, \ldots, m\}$ define $S_i$ and $Q_i$ as follows. If there does
not exist $(y_2^2,y_2^3) \in Y$
such that some $x \in X(T_i)$ with $c(x)=1$ is complete to $\{y_2^2,y_2^3\}$, 
set  $S_i=Q_i=\emptyset$. If such a pair  exists, let
$(y_2^2,y_2^3)$ be a pair with this property and subject to 
that with the set $N(y_2^2) \cap X(T_i)$ minimal;
let $x \in X(T_i)$ be complete to $\{y_2^2,y_2^3\}$ 
and with  $c(x)=1$; and  set $Q_i=\{(y_2^2,y_2^3)\}$ and $S_i=\{x\}$.
Let $Q=((S_1,Q_1), \ldots, (S_m,Q_m))$.
We claim that $c$ is a precoloring extension of $P_Q$. Write
$P_Q=(G,S',X_0',X', Y',f')$. We need to show that
$c(v)=f'(v)$ for every $v \in S' \cup X_0'$. Since $c$ is a precoloring extension of $P$, it follows that $c(v)=f(v)=f'(v)$ for every $v \in S \cup X_0$.
Since $S' \setminus S = \bigcup_{s=1}^mS_s$ and $c(v)=f'(v)=1$ for
every $v \in  \bigcup_{s=1}^mS_s$, we deduce that $c(v)=f'(v)$ for every 
$v \in S'$. Finally let $v \in X_0' \setminus X_0$. Then $v \in X$,
$f'(v)$ is the unique color of $M_P(v) \setminus \{1\}$, and there
are three possibilities. 
\begin{enumerate}
\item $1 \in M_P(v)$ and $v$ has a neighbor in $\bigcup_{s=1}^mS_s$, or
\item there is $i \in \{1, \ldots, m\}$ with $S_i=\{x_i\}$ and $Q_i=\{(y_i^2,y_i^3)\}$, and there exists   $(y_2,y_3) \in Y$ 
such $N(y_2) \cap X(T_i) \subseteq X(T_i) \cap (N(y_i^2) \setminus \{x_i\})$, or
\item there is $i \in \{1, \ldots, m\}$ with $S_i=Q_i=\emptyset$,
and there exists $(y_2,y_3) \in Y$ such that $v \in X(T_i) \cap N(y_2)$.
\end{enumerate}
We show that in all these cases $c(v)=f'(v)$.
\begin{enumerate}
\item Let $x \in  \bigcup_{s=1}^mS_s$. Then $c(x)=1$, and so 
$c(v) \neq 1$, and thus $c(v)=f'(v)$.
\item By the choice of $y_i^2,y_i^3$ and since 
$N(y_2) \cap X(T_i) \subseteq (N(y_i^2) \cap X(T_i)) \setminus \{x_i\})$, it 
follows that $c(u) \neq 1$ for every  $u \in N(y_2) \cap X(T_i)$,
and  therefore $c(v)=f'(v)$.
\item Since $S_i=\emptyset$, it follows that for every  
$(y_2,y_3) \in Y$ and for every $u \in N(y_2) \cap X(T_i)$ we have
$c(u) \neq 1$,  and again $c(v)=f'(v)$.
\end{enumerate}
This proves that $c$ is an extension of $P_Q$, and completes the proof of 
Lemma~\ref{tidy}.
\end{proof}

Repeatedly applying Lemma~\ref{tidy} and using symmetry, we deduce the 
following:

\begin{lemma}
\label{tidy2} 
 There is a function $q : \mathbb{N} \rightarrow \mathbb{N}$ such that
the following holds.  Let $G$ be a $P_6$-free graph.  Let $P=\esspc$
be a clean excellent starred precoloring of $G$.  Then there 
is an
algorithm with running time $O(|V(G)|^{q(|S|)})$ that outputs a
collection $\mathcal{L}$ of excellent starred precolorings of $G$ such
that:
\begin{itemize}
\item $|\mathcal{L}| \leq |V(G)|^{q(|S|)}$;
\item $|S'| \leq q(|S|)$ for every $P' \in \mathcal{L}$; 
\item every $P' \in \mathcal{L}$ is clean and tidy;
\item $\mathcal{L}$ is an equivalent collection for $P$.
\end{itemize}
\end{lemma}

Our next goal is to show that a clean and tidy precoloring can be replaced 
with an  equivalent collection of orderly precolorings.

\begin{lemma}
\label{orderly} 
There is a function $q : \mathbb{N} \rightarrow \mathbb{N}$ such that
the following holds.  Let $G$ be a $P_6$-free graph.  Let $P=\esspc$
be a clean, tidy starred precoloring of $G$.  Then there is an
algorithm with running time $O(|V(G)|^{q(|S|)})$ that outputs a
collection $\mathcal{L}$ of excellent starred precolorings of $G$ such
that:
\begin{itemize}
\item $|\mathcal{L}| \leq |V(G)|^{q(|S|)}$;
\item $|S'| \leq q(|S|)$ for every $P' \in \mathcal{L}$;
\item every $P' \in \mathcal{L}$ is clean and tidy;
\item every $P' \in \mathcal{L}$ is $kl$-orderly for every $(k,l)$ for which
$P$ is $kl$-orderly;
\item every $P' \in \mathcal{L}$ is $14$-orderly;
\item $P$ is equivalent to $\mathcal{L}$.
\end{itemize}
\end{lemma}

\begin{proof}
Without loss of generality we may assume that $X_0=X^0(P)$,
and so $L_P(x)=M_P(x)$ for every $x \in X$.
We may assume that $P$ is not $14$-orderly for otherwise we may set 
$\mathcal{L}=\{P\}$.
Let $Y=\{y \in Y^* \text{ such that } \{2,3\} \subseteq M_P(y)\}$.
Let $T_1, \ldots, T_p$ be the types with $L(T_s)=\{1,2\}$ and
$T_{p+1}, \ldots, T_m$ the types with $L(T_s)=\{1,3\}$.
Let $\mathcal{Q}$ be the collection of all $p(m-p)$-tuples of pairs 
$(S_i,Q_j)$ with $i \in \{1, \ldots, p\}$ and
$j \in \{p+1, \ldots,m\}$, where  
\begin{itemize}
\item  $S_i,Q_j \subseteq Y$;
\item $|S_i|,|Q_j| \in \{0,1\}$;
\item if $N(S_i) \cap X(T_i) = \emptyset$, then $S_i=\emptyset$;
\item if $N(Q_j) \cap X(T_j)=\emptyset$, then $Q_j=\emptyset$.
\end{itemize}

For $Q \in \mathcal{Q}$ construct a precoloring $P_Q$ as follows.
Let $i \in \{1, \ldots, p\}$ and $j \in \{p+1, \ldots, m\}$.
\begin{itemize}
\item Assume first   that $S_i=\{y_i\}$ $Q_j=\{y_j\}$. 
If there is an edge between $N(y_i) \cap X(T_i)$ and 
$N(y_j) \cap X(T_j)$, remove $Q$ from $\mathcal{Q}$.
Now suppose that $N(y_i) \cap X(T_i)$ is anticomplete to 
$N(y_j) \cap X(T_j)$.
Move  $T=(N(y_i) \cap X(T_i)) \cup (N(y_j) \cap X(T_j))$
into $X_0$ with color $1$. For every $y \in Y$ complete to $T$ and both with a 
neighbor in $X(T_i) \setminus T$ and a neighbor in $X(T_j) \setminus T$, proceed
as follows: if $4 \in M_P(y)$, move $y$ to 
$X_0$ with color $4$; if $4 \not \in M_P(y)$, remove $Q$ from $\mathcal{Q}$.
\item Next assume that exactly one of $S_i,Q_j$ is non-empty.
By symmetry we may assume that $S_i=\{y_i\}$ and $Q_j=\emptyset$.
Move $T=N(y_i) \cap X(T_i)$  into $X_0$ with color  $1$. 
For every $y \in Y$ complete to $T$ and both with a neighbor
in $X(T_i) \setminus T$ and a neighbor in $X(T_j)$, proceed as follows:
if $4 \in M_P(y)$, move $y$ to $X_0$
with color $4$; if $4 \not \in M_P(y)$, remove $Q$ from $\mathcal{Q}$.
\item Finally assume that $S_i=Q_j=\emptyset$.
For every $y \in Y$ with both a neighbor in $X(T_i)$ and a neighbor in
$X(T_j)$, proceed as follows: if $4 \in M_P(y)$, move $y$ to $X_0$ with color 
$4$; if $4 \not \in M_P(y)$, remove $Q$ from $\mathcal{Q}$.
\end{itemize}
Let $Q \in \mathcal{Q}$, and let $P_Q=(G,S', X_0',X',{Y^*}',f')$.
Since $X' \subseteq X$, $Y' \subseteq Y^*$ and
$M_{P_Q}(v) \subseteq M_P(v)$ for every $v$,  it follows that $P_Q$
is excellent, clean,  tidy, and that for
$k,l \in \{1,2,3,4\}$, if $P$ is $kl$-orderly, then $P_Q$ is 
$kl$-orderly.

Next  we show that $P_Q$ is $14$-orderly. Suppose that some $y \in Y$
has a neighbor in $x_2 \in X'_{12}$ and a neighbor in 
$x_3 \in X'_{13}$ such that $x_2$ is non-adjacent to $x_3$. Then  
$x_2 \in X_{12}$ and $x_3 \in X_{13}$. We may assume that
$x_2 \in X(T_1)$ and $x_3 \in X(T_{p+1})$.  Since $x_2,x_3 \not \in X^0(P_Q)$,
it follows that both $S_1 \neq \emptyset$ and $Q_{p+1} \neq \emptyset$.
Let $S_1=\{y_2\}$ and $Q_{p+1}=\{y_3\}$. Since $x_2,x_3 \not \in X^0(P_Q)$,
it follows that $y_2$ is non-adjacent to $x_2$, and $y_3$ is non-adjacent to 
$x_3$.  Since $y \not \in X^0(P_Q)$, we may assume by symmetry that there is 
$x_2' \in N(y_2) \cap X(T_1)$ such that $y$ is non-adjacent to $x_2'$.  
Let $x_3' \in  N(y_3) \cap X(T_{p+1})$.
Since $x_2, x_3, y \not \in X^0(P_Q)$,
it follows that $\{x_2',x_3'\}$ is anticomplete to $\{x_2,x_3\}$. By the 
construction of $Q$, $x_2'$ is non-adjacent to $x_3'$. 
By Lemma~\ref{Sophiespath}, $y$ is 
adjacent to $x_3'$. Since $L_P(T_1)=\{1,2\}$, there is $s_3 \in S$ complete 
to  $X(T_1)$. Since  
$3 \in M_{P_Q}(y) \cap L_P(y_2) \cap L_P(y_3) \cap L_P(x_3') \cap L(x_3)$,
it follows that $s_3$ is anticomplete to $\{y,y_2,y_3,x_3,x_3'\}$. 
Similarly, since $L_P(T_{p+1})=\{1,3\}$, there is $s_2 \in S$ complete 
to  $X(T_{p+1})$. Since  
$2 \in M_{P_Q}(y) \cap L_P(y_2) \cap L_P(y_3) \cap L_P(x_2) \cap L_P(x_2')$,
it follows that $s_2$ is anticomplete to $\{y,y_2,y_3,x_2,x_2'\}$.
Since
$y_2-x_2'-s_3-x_2-y-t$ is not a $P_6$ for $t \in \{x_3,x_3'\}$, it follows that 
$y_2$ is complete to  $\{x_3,x_3'\}$. Since $y_3-x_3'-y-x_2-s_3-x_2'$ is not a $P_6$, it follows that $y_3$
is adjacent to at least one of $x_2,x_2'$. Since the path $x_2-y-x_3-y_2-x_2'$
cannot be extended to a $P_6$ via $y_3$, follows that $y_3$ is complete
to $\{x_2,x_2'\}$. But now $s_2-x_3-y-x_2-y_3-x_2'$ is a $P_6$, a contradiction.
This proves that $P_Q$ is $14$-orderly.

Observe that $S'=S$, and so $|S'|=|S|$.
Observe also that also that $p(m-p) \leq {(\frac{m}{2})}^2$, and 
since $m \leq 2^{|S|}$, it follows that $p(m-p) \leq  2^{2|S|-2}$.
Let $\mathcal{L}=\{P_Q \; : \; Q \in \mathcal{Q}\}$.
Now $|\mathcal{L}| \leq |V(G)|^{2p(m-p)} \leq |V(G)|^{2^{2|S|-1}}$.

We show that $\mathcal{L}$ is an equivalent collection for $P$. Since every 
$P' \in \mathcal{L}$ is obtained from $P$ by precoloring some vertices and 
updating, it is clear that if $c$ is a precoloring extension of a member of 
$\mathcal{L}$, then $c$ is a precoloring extension of $P$.
To see the converse, suppose that $P$ has a precoloring 
extension $c$.  For every
$i \in \{1, \ldots, p\}$ and $j \in \{p+1, \ldots, m\}$
 define $S_i$ and $Q_j$ as follows. If every vertex
of $Y$ has a neighbor $x \in X(T_i)$ with $c(x) \neq 1$, set 
$S_i=\emptyset$, and if every  vertex
of $Y$ has a neighbor $x \in X(T_j)$ with $c(x) \neq 1$, set 
$Q_j=\emptyset$.
 If some vertex of $Y$ has no neighbor $x \in X(T_i)$
with $c(x) \neq 1$, let $y_i$ be a vertex with this property and in addition 
with  $N(y) \cap X(T_i)$ maximal; set $S_i=\{y_i\}$.
If some vertex of $Y$ has no neighbor $x \in X(T_j)$
with $c(x) \neq 1$, let $y_j$ be  a vertex with this property and in addition 
with  $N(y) \cap X(T_j)$ maximal; set $Q_j=\{y_j\}$.
We claim that $c$ is a precoloring extension of $P_Q$. Write
$P_Q=(G,S',X'_0,X', Y',f')$. We need to show that
$c(v)=f'(v)$ for every $v \in S' \cup X_0'$. Since $c$ is a precoloring extension of $P$, and since $S=S'$, it follows that $c(v)=f(v)=f'(v)$ for every 
$v \in S' \cup X_0$. Let $v \in X_0' \setminus X_0$. 
It follows that either 
\begin{enumerate}
\item $S_i=\{y_i\}$, $Q_j=\{y_j\}$, and
$v \in X$ and $v \in (N(y_i) \cap X(T_i)) \cup (N(y_j) \cap X(T_j))$
and $f'(v)=1$, or
\item  $S_i=\{y_i\}$, $Q_j=\{y_j\}$, $v \in Y$, $v$ is complete
to $(N(y_i) \cap X(T_i)) \cup (N(y_j) \cap X(T_j))$,
$v$ has both a neighbor in $X(T_i) \setminus N(y_i)$
and a neighbor in $X(T_j) \setminus N(y_j)$, and $f'(v)=4$,
or
\item  (possibly with the roles of $i$ and $j$ exchanged)
$S_i=\{y_i\}$, $Q_j=\emptyset$, and
$v \in X$ and $v \in N(y_i) \cap X(T_i)$,
and $f'(v)=1$, or
\item  (possibly with the roles of $i$ and $j$ exchanged)
$S_i=\{y_i\}$, $Q_j=\emptyset$, $v \in Y$, $v$ is complete
to $N(y_i) \cap X(T_i)$,
$v$ has both a neighbor in $X(T_i) \setminus N(y_i)$
and a neighbor in $X(T_j)$, and $f'(v)=4$, or
\item $S_i=Q_j=\emptyset$, $v \in Y$, 
$v$ has both a neighbor in $X(T_i)$ and a neighbor in $X(T_j)$, and
$f'(v)=4$.
\end{enumerate}
We show that in all these cases $c(v)=f'(v)$.
\begin{enumerate}
\item By the choice of $y_i,y_j$, 
$c(u)=1$ for every  $u \in (N(y_i) \cap X(T_i)) \cup (N(y_j) \cap X(T_j))$,
and so $c(v)=f'(v)$.
\item It follows from the maximality of $y_i,y_j$ that $v$ has
both a neighbor $x_2 \in X(T_i)$ with $c(x_2)=2$ 
and a neighbor $x_3 \in X(T_j)$ with $c(x_3)=3$.
Since $P$ is clean, it follows that $1 \not \in M_P(v)$, and  therefore
$c(v)=4$.
\item By the choice of $y_i$, 
$c(u)=1$ for every  $u \in N(y_i) \cap X(T_i)$,  and so $c(v)=f'(v)$.
\item It follows from the maximality of $y_i$ that $v$ has 
a neighbor $x_2 \in X(T_i)$ with $c(x_2)=2$. Since $Q_j=\emptyset$, 
$v$ has a neighbor $x_3 \in X(T_j)$ with $c(x_3)=3$. Since $P$ is clean,
it follows that $1 \not \in M_P(v)$, and so  $c(v)=4$.
\item Since $S_i=Q_j=\emptyset$, it follows that 
$v$ has  both
a neighbor $x_2 \in X(T_i)$ with $c(x_2)=2$,  
and a neighbor $x_3 \in X(T_j)$ with $c(x_3)=3$.
Since $P$ is clean, it follows that $1 \not \in M_P(v)$, and so
$c(v)=4$.
\end{enumerate}
This proves that $c$ is an extension of $P_Q$, and completes the proof of 
Lemma~\ref{orderly}.
\end{proof}

Repeatedly applying Lemma~\ref{orderly} and using symmetry, we deduce the 
following:

\begin{lemma}
\label{orderly2} 
There  is a function $q : \mathbb{N} \rightarrow \mathbb{N}$ such that the following holds. 
Let $G$ be a $P_6$-free graph. 
Let $P=\esspc$ be a clean and tidy excellent starred precoloring of $G$.
Then there is an algorithm with running time $O(|V(G)|^{q(|S|)})$ that outputs a collection
  $\mathcal{L}$ of excellent starred  precolorings of $G$ such that:
\begin{itemize}
\item $|\mathcal{L}| \leq |V(G)|^{q(|S|)}$;
\item $|S'| \leq q(|S|)$ for every $P' \in \mathcal{L}$; 
\item every $P' \in \mathcal{L}$ is clean, tidy and orderly;
\item $P$ is equivalent to $\mathcal{L}$.
\end{itemize}
\end{lemma}

Next we show that a clear, tidy and orderly excellent starred
precoloring can be replaced by an equivalent collection of spotless
precolorings.

\begin{lemma}
\label{spotless}
There is a function $q : \mathbb{N} \rightarrow \mathbb{N}$ such that
the following holds.  Let $G$ be a $P_6$-free graph.  Let $P=\esspc$
be a clean, tidy and orderly excellent starred precoloring of $G$.
Then there is an algorithm with running time $O(|V(G)|^{q(|S|)})$ that
outputs a collection $\mathcal{L}$ of excellent starred precolorings
of $G$ such that:
\begin{itemize}
\item $|\mathcal{L}| \leq |V(G)|^{q(|S|)}$;
\item $|S'| \leq q(|S|)$ for every $P' \in \mathcal{L}$; 
\item every $P' \in \mathcal{L}$ is clean, tidy and orderly;
\item every $P' \in \mathcal{L}$ is $kl$-spotless for every $(k,l)$ for which
$P$ is $kl$-spotless;
\item every $P' \in \mathcal{L}$ is $14$-spotless;
\item $P$ is equivalent to $\mathcal{L}$.
\end{itemize}
\end{lemma}

\begin{proof}
The proof follows closely the proof of Lemma~\ref{orderly}, deviating from it 
only when we show that  every  $P' \in \mathcal{L}$ is $14$-spotless.
Without loss of generality we may assume that $X_0=X^0(P)$,
and so $L_P(x)=M_P(x)$ for every $x \in X$.
We may assume that $P$ is not $14$-spotless for otherwise we may set 
$\mathcal{L}=\{P\}$.
Let $Y=\{y \in Y^* \text{ such that } \{2,3\} \subseteq M_P(y)\}$.
Let $T_1, \ldots, T_p$ be the types with $L(T_s)=\{1,2\}$ and
$T_{p+1}, \ldots, T_m$ the types with $L(T_s)=\{1,3\}$.
Let $\mathcal{Q}$ be the collection of all $p(m-p)$-tuples 
$(P_i,Q_j)$ with $i \in \{1, \ldots, p\}$ and
$j \in \{p+1, \ldots,m\}$, where $S_i,Q_i \subseteq Y$ and 
$|P_i|,|Q_i| \in \{0,1\}$. 

For $Q \in \mathcal{Q}$ construct a precoloring $P_Q$ as follows.
Let $i \in \{1, \ldots, p\}$ and $j \in \{p+1, \ldots, m\}$.
\begin{itemize}
\item Assume first   that $S_i=\{y_i\}$ $Q_j=\{y_j\}$. 
If there is an edge between $N(y_i) \cap X(T_i)$ and 
$N(y_j) \cap X(T_j)$, remove $Q$ from $\mathcal{Q}$.
Now suppose that $N(y_i) \cap X(T_i)$ is anticomplete to 
$N(y_j) \cap X(T_j)$.
Move  $T=(N(y_i) \cap X(T_i)) \cup (N(y_j) \cap X(T_j))$
into $X_0$ with color 
$1$. For every $y \in Y$ complete to $T$ and both with a neighbor
in $X(T_i) \setminus T$ and a neighbor in $X(T_j) \setminus T$, proceed as
follows: if $4 \in M_P(y)$,  move $y$ to $X_0$ with color $4$; if 
$4 \not \in M_P(y)$, remove $Q$ from $\mathcal{Q}$.
\item Next assume that exactly one of $S_i,Q_j$ is non-empty.
By symmetry we may assume that $S_i=\{y_i\}$ and $Q_j=\emptyset$.
Move $T=N(y_i) \cap X(T_i)$  into $X_0$ with color  $1$. 
For every $y \in Y$ complete to $T$ and both with a neighbor
in $X(T_i) \setminus T$ and a neighbor in $X(T_j)$, proceed as follows. If
$4 \in M_P(y)$, move $y$ to $X_0$
with color $4$; if $4 \not \in M_P(y)$, remove $Q$ from $\mathcal{Q}$.
\item Finally assume that $S_i=S_j=\emptyset$.
For every $y \in Y$ with both a neighbor in $X(T_i)$ and a neighbor in
$X(T_j)$, proceed as follows: if $4 \in M_P(y)$, move $y$ to $X_0$ with color $4$; if $4 \not \in M_P(y)$, remove $Q$ from $\mathcal{Q}$.
\end{itemize}
Let $Q \in \mathcal{Q}$, and let $P_Q=(G,S', X_0',X_0,{Y^*}',f')$.
If $f'$ is not a proper coloring of $G|(S' \cup X_0')$, remove $Q$ from 
$\mathcal{Q}$.
Since $X' \subseteq X$, $Y' \subseteq Y^*$ and
$M_{P_Q}(v) \subseteq M_P(v)$ for every $v$,  it follows that $P_Q$
is excellent, clean, tidy and orderly, and that for
$k,l \in \{1,2,3,4\}$, if $P$ is $kl$-spotless, then $P_Q$ is 
$kl$-spotless.

Next  we show that $P_Q$ is $14$-spotless. Suppose that some $y \in Y$
has a neighbor in $x_2 \in X'_{12}$ and a neighbor in 
$x_3 \in X'_{13}$. Then  
$x_2 \in X_{12}$ and $x_3 \in X_{13}$. We may assume that
$x_2 \in X(T_1)$ and $x_3 \in X(T_{p+1})$.  Since $x_2,x_3 \not \in X^0(P_Q)$,
it follows that both $S_1 \neq \emptyset$ and $Q_{p+1} \neq \emptyset$.
Let $S_1=\{y_2\}$ and $Q_{p+1}=\{y_3\}$. Since $x_2,x_3 \not \in X^0(P_Q)$,
it follows that $y_2$ is non-adjacent to $x_2$, and $y_3$ is non-adjacent to 
$x_3$.  Since $y \not \in X^0(P_Q)$, we may assume by symmetry that there is 
$x_2' \in N(y_2) \cap X(T_1)$ such that $y$ is non-adjacent to $x_2'$.  
Let $x_3' \in  N(y_3) \cap X(T_{p+1})$.
Since $x_2, x_3 \not \in X^0(P_Q)$,
it follows that $\{x_2',x_3'\}$ is anticomplete to $\{x_2,x_3\}$. 
By the  construction of $Q$, $x_2'$ is non-adjacent to $x_3'$. 
Now, since $G$ is orderly, $y$ is non-adjacent to $x_3'$, 
contrary to Lemma~\ref{Sophiespath}.
This proves that $P_Q$ is $14$-spotless.

Observe that $S=S'$, and so $|S|=|S'|$.
Observe also that also that $p(m-p) \leq {(\frac{m}{2})}^2$, and 
since $m \leq 2^{|S|}$, it follows that $p(m-p) \leq  2^{2|S|-2}$.
Let $\mathcal{L}=\{P_Q \; : \; Q \in \mathcal{Q}\}$.
Now $|\mathcal{L}| \leq |V(G)|^{2p(m-p)} \leq |V(G)|^{2^{2|S|-1}}$.

The remainder of the proof follows word for word the proof of Lemma~\ref{orderly}, and we omit it.
This proves that $P_Q$ has a precoloring extension, and completes the proof of
Lemma~\ref{spotless}.
\end{proof}

Observe that if an excellent starred precoloring is spotless, then it is
clean and orderly.
Repeatedly applying Lemma~\ref{spotless} and using symmetry, we deduce the 
following:

\begin{lemma}
  \label{spotless2} There is a function
  $q : \mathbb{N} \rightarrow \mathbb{N}$ such that the following
  holds.  Let $G$ be a $P_6$-free graph.  Let $P=\esspc$ be a clean,
  tidy and orderly excellent starred precoloring of $G$.  Then there
  is an algorithm with running time $O(|V(G)|^{q(|S|)})$ that outputs
  a collection $\mathcal{L}$ of excellent starred precolorings of $G$
  such that:
\begin{itemize}
\item $|\mathcal{L}| \leq |V(G)|^{q(|S|)}$;
\item $|S'| \leq q(|S|)$ for every $P' \in \mathcal{L}$;
\item every $P' \in \mathcal{L}$ is tidy and spotless;
\item $P$ is equivalent to $\mathcal{L}$.
\end{itemize}
\end{lemma}

We now summarize what we have proved so far. 
Let $P=\esspc$ be an excellent starred precoloring of a $P_6$-free graph 
$G$. We say that $y \in Y^*$ is 
\emph{wholesome} if $|M_P(y)| \geq 3$. A component of $G|Y^*$ if \emph{wholesome} if it contains a wholesome vertex.
We say that $P$ is \emph{near-orthogonal} if for every 
wholesome $y \in Y^*$ either 
\begin{itemize}
\item $y$ has orthogonal neighbors in $X$, or
\item there exist $\{i,j,k,l\}=\{1,2,3,4\}$ such that 
\begin{itemize}
\item $N(y) \cap X \subseteq X_{ki} \cup X_{kj}$, and 
\item For every $u \in C(y)$, $|M_P(u) \cap \{i,j\}| \leq 1$, and
\item   if there is $v_i \in C(y)$ with $i \in M_P(v_i)$ and 
$v_j \in C(y)$ with $j \in M_P(v_j)$, then for some $u \in C(y)$,
$l \not \in M_P(u)$.
\end{itemize}
\end{itemize}

\begin{lemma}
\label{near} There is a function $q : \mathbb{N} \rightarrow \mathbb{N}$ such that the following holds. 
Let $P=\esspc$ be an excellent starred  precoloring of a $P_6$-free graph 
$G$. Then there
  is an algorithm with running time $O(|V(G)|^{q(|S|)})$ that outputs a collection
$\mathcal{L}$ of excellent starred  precolorings of $G$ such that:
\begin{itemize}
\item $|\mathcal{L}| \leq |V(G)|^{q(|S|)}$;
\item $|S'| \leq q(|S|)$ for every $P' \in \mathcal{L}$; 
\item every $P' \in \mathcal{L}$ is near-orthogonal;
\item $P$ is equivalent to $\mathcal{L}$.
\end{itemize}
\end{lemma}

\begin{proof}
Let $\mathcal{L}_1$ be the collection of precolorings obtained by
applying Lemma~\ref{clean2} to $P$.
Let $\mathcal{L}_2$ be the union of the collections of precolorings obtained by
applying Lemma~\ref{tidy2} to each member of $\mathcal{L}_1$.
Let $\mathcal{L}_3$ be the union of the collections of precolorings obtained by
applying Lemma~\ref{orderly2} to each member of $\mathcal{L}_2$.
Let $\mathcal{L}$ be the union of the collections of precolorings obtained by applying Lemma~\ref{spotless2} to each member of $\mathcal{L}_3$.
Then $\mathcal{L}$ satisfies the 
first, second and fourth bullet in the statement of Lemma~\ref{near}, and
every $P' \in \mathcal{L}$ is tidy  and spotless. 
Let $P' \in \mathcal{L}$, write $P'=(S',X_0',X', Y',f')$. 
Suppose that  $P'$ is not near-orthogonal. 
Let $y \in Y'$, and assume that the 
neighbors of $y$ are not orthogonal.
We show that $y$ satisfies  the conditions in the definition of near-orthogonal.
We may assume that $y$ has a neighbor in $X'_{12}$ and a neighbor in $X'_{13}$.
Since $P'$ is spotless, it follows that for every $u \in C(y)$, 
$|M_P(u) \cap \{2,3\}| \leq 1$. 
Since $y$ is wholesome, we may assume
that $M_P(y)=\{1,2,4\}$. Since $P'$ is spotless, it follows that
$N(y) \cap X' \subseteq X'_{12} \cup X'_{13}$.
Since $P'$ is tidy and $1 \in M_P(y)$, it follows 
that if there is $v_2 \in C(y)$ with $2 \in M_P(v_2)$ and 
$v_3 \in C(y)$ with $3 \in M_P(v_3)$, then for some $u \in C(y)$
$4 \not \in M_P(u)$. This proves that $y$ satisfies the conditions in
the definition of near orthogonal, and completes the proof of
Lemma~\ref{near}.
\end{proof}

Let $P=\esspc$ be an excellent  starred precoloring.
Let $\{i,j,k,l\}=\{1,2,3,4\}$, let $T^i$ be a type of $X$
with $L_P(T^i)=\{i,k\}$ and let $T^j$ be a type of $X$
with $L_P(T^j)=\{j,k\}$.
A \emph{type A extension with respect to $(T^i,T^j)$} is  a precoloring 
extension $c$ 
of $P$ such that there exists $y \in Y^*$ with $k,i \in M_P(y)$
and such that $y$ has a neighbor $x_i \in X(T^i)$ and a neighbor
$x_j \in X(T^j)$ with $c(x_i)=c(x_j)=k$.

Let $\mathcal{T}(P)$ be the set of all pairs $(T^i,T^j)$ of types of $X$ with 
$|L_P(T^j) \cap L_P(T^j)|=1$.
We say that $P$ is \emph{smooth}
if $P$ has a precoloring extension $c$ such that for every 
$(T^i,T^j) \in \mathcal{T}(P)$, $c$ is not of type A with respect to 
$(T^i,T^j)$. A precoloring extension of $P$ is \emph{good} if it is not of
type A for any $T \in \mathcal{T}(P)$.

We say that an excellent  starred precoloring $P'=(G,S',X_0',X',{Y^*}',f')$ is 
a \emph{refinement} of $P$ if for every  type $T'$  of $X'$,  there is a type 
$T$ of $X$ such that $X'(T') \subseteq X(T)$. 

\begin{lemma}
\label{smooththm}
There  is a function $q : \mathbb{N} \rightarrow \mathbb{N}$ 
such that the following holds.
Let $P=\esspc$ be a near-orthogonal   excellent starred  
precoloring of a $P_6$-free  graph $G$. There is an 
algorithm with running time $O(|V(G)|^{q(|S|)}$
that outputs a collection
$\mathcal{L}$ of near-orthogonal  excellent starred   precolorings of 
$G$  such that:
\begin{itemize}
\item $|\mathcal{L}| \leq |V(G)|^{q(|S|)}$;
\item $|S'| \leq q(|S|)$ for every $P' \in \mathcal{L}$;
\item a precoloring extension of a member of $\mathcal{L}$ is also a 
precoloring  extension of $P$;
\item if $P$ has a precoloring extension, then some $P' \in \mathcal{L}$ is 
smooth.
\end{itemize}
\end{lemma}

\begin{proof}
Let $\mathcal{T}(P)=\{(T_1,T_1'), \ldots, (T_t,T_t')\}$.
Let $\mathcal{Q}$ be the collection of $t$-tuples of 
triples $Q_{T_i,T_i'}=(Y_{T_i,T_i'}, A_{T_i,T_i'}, B_{T_i,T_i'})$ such that
\begin{itemize}
\item $|Y_{T_i,T_i'}|= |A_{T_i,T_i'}|=|B_{T_i,T_i'}| \leq 1$.
\item $A_{T_i,T_i'} \subseteq X(T_i)$.
\item $B_{T_i,T_i'} \subseteq X(T_i')$.
\item $Y_{T_i,T_i'} \subseteq Y^*$ and if $Y_{T_i,T_i'}=\{y\}$, then 
$L_P(T_i) \subseteq M_P(y)$.
\item $Y_{T_i,T_i'}$ is complete to $A_{T_i,T_i'} \cup B_{T_i,T_i'}$.
\item $A_{T_i,T_i'}$ is anticomplete to $B_{T_i,T_i'}$.
\end{itemize}
For $Q=(Q_{T_i,T_i'})_{(T_i,T_i') \in \mathcal{T}(P)} \in \mathcal{Q}$, 
we construct a precoloring $P_Q$ by moving
$A_{T_i,T_i'} \cup B_{T_i,T_i'}$  to the seed with the unique color
of $L_P(T_i) \cap L_P(T_i')$ for all $(T_i,T_i') \in \mathcal{T}(P)$.
Let $P_Q=(G,S',X_0',X',Y',f')$. 
Since $X' \subseteq X$ and $Y' \subseteq Y^*$, and
$M_{P'}(v) \subseteq M_P(v)$ for every $v \in V(G)$,
it follows that $P_Q$
is excellent, near-orthogonal  and  for every type $T'$ of $X'$,
there is a type $T$ of $X$ such that  $X'(T') \subseteq X(T)$.

Let $\mathcal{L}=\{P\} \cup \{P_Q \; : \; Q \in \mathcal{Q}\}$.
Observe  that there are at most 
$2^{|S|}$ types, and therefore $t \leq  2^{2|S|}$.
Now  $|S'| \leq |S|+2t \leq |S|+2^{2|S|+1}$ and
$|\mathcal{L}| \leq |V(G)|^{3t} \leq |V(G)|^{3 \times 2^{2|S|}}$.

Since every member of $\mathcal{L}$ is obtained from $P$ by precoloring
some vertices and updating, it follows that  every precoloring extension of a 
member of $\mathcal{L}$ is also a precoloring extension of $P$. 

Now we prove the last assertion of Lemma~\ref{smooththm}.
Suppose that $P$ has a precoloring extension. 
We need to show that some $P' \in \mathcal{L}$ is smooth. 
Let $c$ be a precoloring extension of $P$.
For every $(T_i,T_i') \in \mathcal{T}(P)$ such that $c$ is 
of type A with respect to  $(T_i,T_i')$, proceed as follows.
We may assume that $L_P(T_i)=\{1,2\}$ and $L_P(T_i')=\{1,3\}$.
Let  $y \in Y^*$ with $1,2 \in M_P(y)$, $x_2 \in X(T_i)$ and
$x_3 \in X(T_i')$ such that $y$ is adjacent to $x_2,x_3$ and  $c(x_2)=c(x_3)=1$,
and subject to the existence of such $x_2,x_3$, choose $y$ with the set
$\{x \in N(y) \cap X(T_i') \text{ such that } c(x)=1\}$ minimal.
Let $Q_{T_i,T_i'}=(\{y\},\{x_2\},\{x_3\})$. 
For every $(T_i,T_i') \in \mathcal{T}(P)$ such that $c$ is not
of type A with respect to  $(T_i,T_i')$, set
$Q_{T_i,T_i'}=(\emptyset, \emptyset, \emptyset)$. 
Let $Q=(Q_{T_i,T_i'})_{(T_i,T_i') \in \mathcal{P}}$; then $P_Q \in \mathcal{L}$.

We claim that
$c$ is a precoloring extension of $P_Q$ that is not of type A for any
$(T_i,T_i') \in \mathcal{T}(P_Q)$.
Write $P_Q=(G,S',X_0',X',Y',f')$. 
Let $\{i,j,k,l\}=\{1,2,3,4\}$.
Suppose that  $T^i$ is a type of $X'$ with $L_{P_Q}(T^i)=\{i,k\}$ 
and  $T^j$ is a type of $X'$ with $L_{P_Q}(T^j)=\{j,k\}$,  and such that  
$(T^i,T^j) \in \mathcal{T}(P_Q)$,
and $y' \in Y'$ with $i,k \in M_{P_Q}(y')$ has neighbor
$x_i' \in X'(T^i)$ and $x_j' \in X'(T^j)$ with $c(x_i')=c(x_j')=k$.
Let $(\tilde{T^i},\tilde{T^j}) \in \mathcal{T}(P)$ be 
such that $X'(T^i) \subseteq X(\tilde{T^i})$ and
$X'(T^j) \subseteq X(\tilde{T^j})$.
Since $i,k  \in M_P(y)$, it follows that $c$ is of type A for 
$(\tilde{T^i},\tilde{T^j})$, and therefore
$|Y_{\tilde{T^i},\tilde{T^j}}|=|A_{\tilde{T^i},\tilde{T^j}}|=|B_{\tilde{T^i},\tilde{T^j}}|=1$.
Let 
$Y_{\tilde{T^i},\tilde{T^j}}=\{y\}$
$A_{\tilde{T^i},\tilde{T^j}}=\{x_i\}$
and $B_{\tilde{T^i},\tilde{T^j}}=\{x_j\}$.
Since $k  \in M_{P_Q}(y')$ it follows that $y'$ is anticomplete to 
$\{x_i,x_j\}$.
By the choice of $y$,  it follows that $y'$ has a neighbor 
$x' \in X(\tilde{T^j}) \setminus N(y)$ with $c(x')=k$, and so we may assume that
$x_j'$ is non-adjacent to $y$.
Since $L_P(T_i')=\{j,k\}$ there exists $s_i \in S$ with 
$f(s_i)=i$ such that $s_i$ is complete to $\{x_j,x_j'\}$. 
Since $i \in L_P(x_i) \cap L_P(y') \cap L_P(y)$, it follows that
$s_i$ is anticomplete to $\{x_i,y',y\}$. Since $c(x_i)=c(x_i')=c(x_j)=c(x_j')$,
it follows that $\{x_i,x_i',x_j,x_j'\}$ is a stable set. 
But now $x_i-y-x_j-s_i-x_j'-y'$ is a $P_6$ in $G$, a contradiction.
This proves that $c$ is a good precoloring extension of $P_Q$, 
and completes the proof of Lemma~\ref{smooththm}.
\end{proof}

We are finally ready to construct orthogonal precolorings.

\begin{lemma}
\label{orthogonal}
There is a function
  $q : \mathbb{N} \rightarrow \mathbb{N}$ such that the following
  holds.
Let $P=\esspc$ be a near-orthogonal  excellent precoloring of a $P_6$-free 
graph $G$.  There exist an induced  subgraph $G'$ of $G$ and an orthogonal 
excellent starred precoloring 
$P'=(G',S',X_0',X',Y',f')$ of $G'$, such that 
\begin{itemize}
\item $S=S'$,
\item  if $P$ is smooth, then $P'$ has a precoloring extension, and  
\item  if $c$ is a precoloring extension of $P'$, then a precoloring extension of $P$ can be constructed from $c$ in polynomial time.
\end{itemize}
Moreover, $P'$ can be constructed in time $O(|V(G)|^{q(|S|)})$.
\end{lemma}

\begin{proof}
We may assume that $P$ is not orthogonal. We say that a component 
$C$ of $G|Y^*$ is \emph{troublesome} if $C$ is wholesome, and the set of attachments of $C$ in $X$ are not orthogonal. Let $W$ be the union of the vertex sets
of the component of $G|Y^*$ that are not wholesome.

We construct a set $Z$, starting with $Z=\emptyset$.
For every troublesome component $C$, proceed as follows.
We may assume that $C$ has attachments in $X_{12}$ and in $X_{13}$.
Since $P$ is near-orthogonal, and $C$ is wholesome, we may assume that $C$ 
contains a vertex $z$ with $M_P(z)=\{1,2,4\}$. 
\begin{itemize}
\item If there is $y \in V(C)$ with $M_P(y)=\{1,3\}$,  move $N(y) \cap X_{12}$
to $X_0$ with color $2$.
\item Suppose that there is no $y$ as in the first bullet.
If $|V(C)|>2$, or $V(C)=\{z\}$ and $z$ has a neighbor $v$ in $X_0$ with
$f(v)=\{4\}$, move $N(z) \cap X_{13}$ to $X_0$ with color $3$.
\item If none of the first two conditions hold, add $V(C)$ to $Z$.
Observe that in this case $V(C)=\{y\}$, $y$ has no neighbors in
$Z \setminus \{y\}$. Moreover, since $P$ is near-orthogonal, $V(C)$ is anticomplete to
$X \setminus (X_{12} \cup X_{13})$, and so  for every $u \in N(y)$, 
$4 \not \in L_P(u)$.
In this case we call $4$ the \emph{free color} of $y$.
\end{itemize}
Let $P''=(G, S',X_0',X'', Y'',f')$  be the precoloring we obtained
after we applied the procedure above to all troublesome components.
Let $G'=G \setminus Z$,
and let $P'=(G',S',X_0',X', Y',f')$ where 
$Y'=Y'' \setminus W \cup Z$ and $X'=X'' \cup W$.
Since no vertex of $W$ is wholesome, It follows from the definition of $M_P$ that
every vertex of $W$ has neighbors of at least two different colors in $S'$ 
(with respect to $f'$).  Since $W$ is anticomplete to $Y'$, 
$X' \setminus W \subseteq X$, and
$Y' \subseteq Y^*$, we deduce that  $P'$ is excellent and  orthogonal.  
It follows from the construction of $Z$
that every precoloring extension of $P'$ can be extended to a 
precoloring extension of $P$ by giving each member of $Z$  its free color.

It remains to show that if $P$ is smooth, then $P'$ has a precoloring extension.
Suppose that $P$ is smooth, and let $c$ be a good precoloring extension of
$P$. We claim that $c|V(G')$ is a  precoloring extension of $P'$. 
We need to show that $c(v)=f'(v)$ for every $v \in S' \cup X_0'$.
Since $S=S'$, and $f(v)=f'(v)$ for every $v \in X_0$, it is enough to show
that $c(v)=f'(v)$ for every $v \in X_0' \setminus X_0$.
Thus we may assume that there is a troublesome component $C$ of $G|Y^*$ 
that has an attachment in $X_{12}$ and an attachment in $X_{13}$, and 
$v \in X(C)$.
Since $P$ is near-orthogonal, we may assume that $C$ contains a vertex $y$ with
$M_P(y)=\{1,2,4\}$, and $v \in X_{12} \cup X_{13}$. There are two possibilities.
\begin{enumerate}
\item There is $y \in V(C)$ with $M_P(y)=\{1,3\}$,
$v \in N(y) \cap X_{12}$ and $f'(v)=2$, but $c(v)=1$. 
We show that this is impossible.
Since $c$ is a good coloring, it follows that $c(u)=3$ for every 
$u \in N(y) \cap X_{13}$,  contrary to the fact that $c$ is a coloring of $G$.
\item There is no $y$ as in the first case, and either $|V(C)|>2$, or 
$V(C)=\{z\}$ and  $z$ has a neighbor $u$ in $X_0$ with $f(u)=\{4\}$,
and $v \in X_{13} \cap N(z)$, $f'(v)=3$ but $c(v)=1$. We show that this too is 
impossible. It follows that there is a vertex
$y' \in V(C)$ with $c(y) \neq 4$. 
Choose such  $y'$ with $4 \not \in M_P(y')$ if  possible.
Since $P$ is excellent, $y'$ is adjacent to $v$.
Since $c$ is a good coloring, it follows that
$c(u) = 2$ for every $u \in X_{12} \cap N(y')$.  This implies that
$c(y')=3$. 
Since $P$ is near-orthogonal and $3 \in M_P(y')$, it follows that 
$2 \not \in M_P(y')$. Since $M_P(y') \neq \{1,3\}$, 
it follows that $4 \in L(y')$.  Since $1,2 \in M_P(y)$ and
$3 \in M_P(y')$, and since $P$ is near-orthogonal, it follows that
there is $z \in V(C)$ such that $4 \not \in M_P(z)$.
Since $c(v)=1$ and $c(u)=2$ for every attachment of $V(C)$ in $X_{12}$,  it 
follows that $c(z)=3$, contrary to the choice of $y'$. 
\end{enumerate}
Thus $c(v)=f'(v)$ for every $v \in S' \cup X_0'$, and so 
$c|V(G')$ is a  precoloring extension of $P'$.
This completes the proof of Lemma~\ref{orthogonal}.
\end{proof}

We can now prove the main result of this section.
\begin{theorem}
\label{orthogonalthm} There is a function $q : \mathbb{N} \rightarrow \mathbb{N}$ such that the following holds. 
Let $P=\esspc$ be an excellent starred  
precoloring of a $P_6$-free  graph $G$ with $|S| \leq C$. 
Then there is an algorithm with running time $O(|V(G)|^{q(|S|)})$ that outputs a collection
$\mathcal{L}$ of orthogonal excellent starred   precolorings of induced 
subgraphs of  $G$  such that:
\begin{itemize}
\item $|\mathcal{L}| \leq |V(G)|^{q(|S|)}$;
\item $|S'| \leq q(|S|)$ for every $P' \in \mathcal{L}$, and
\item $P$ has a precoloring extension, if and only if some 
$P' \in \mathcal{L}$ has a precoloring extension;
\item given a precoloring extension of a member of $\mathcal{L}$, a 
precoloring extension of $P$ can be constructed in polynomial time.
\end{itemize}
\end{theorem}

\begin{proof}

By Lemma~\ref{near} there  exist a function 
$q_1 : \mathbb{N} \rightarrow \mathbb{N}$ 
and a polynomial-time algorithm that outputs a collection
$\mathcal{L}_1$ of excellent starred  precolorings of $G$ such that:
\begin{itemize}
\item $|\mathcal{L}_1| \leq |V(G)|^{q_1(|S|)}$;
\item $|S'| \leq q_1(|S|)$ for every $P' \in \mathcal{L}_1$;
\item every $P' \in \mathcal{L}_1$ is near-orthogonal; and
\item $P$ is equivalent to $\mathcal{L}_1$.
\end{itemize}

Let $P' \in \mathcal{L}_1$. 
Write $P'=(G, S(P'),X_0(P'),X(P'),Y^*(P'), f_{P'})$. 
By Lemma~\ref{smooththm}
there  exist a function $q_2 : \mathbb{N} \rightarrow \mathbb{N}$ 
and a polynomial-time algorithm that outputs a collection
$\mathcal{L}(P')$ of near-orthogonal excellent starred   precolorings of 
$G$  such that:
\begin{itemize}
\item $|\mathcal{L(P')}| \leq |V(G)|^{q_2(|S(P')|)}$;
\item $|S'| \leq q_2(|S(P')|)$ for every $P' \in \mathcal{L}$;
\item if $P'$ has a precoloring extension, then some $P'' \in \mathcal{L}(P')$ 
is  smooth; and
\item a precoloring extension of a member of $\mathcal{L}(P')$ is also a 
precoloring  extension of $P'$.
\end{itemize}
Let $\mathcal{L}_2=\bigcup_{P' \in \mathcal{L}}\mathcal{L}(P')$.

Clearly $\mathcal{L}_2$ has the following properties:

\begin{itemize}
\item $|\mathcal{L}_2| \leq |V(G)|^{q_1(q_2(|S|))}$;
\item $|S'| \leq q_1(q_2(|S(P)|))$ for every $P' \in \mathcal{L}_2$;
\item if $P$ has a precoloring extension, then some $P'' \in \mathcal{L}(P')$ 
is  smooth; and
\item  given a precoloring extension of a member of $\mathcal{L}_2$,
one can construct in polynomial time a precoloring  extension of $P$. 
\end{itemize}

Let $P'' \in \mathcal{L}_2$. Write  
$P''=(G,S(P''),X_0(P;'),X'(P'),Y^*(P''), f_{P''})$.
By Lemma~\ref{orthogonal}, there exists an induced  subgraph $G'$ of $G$ and an 
orthogonal  excellent starred  precoloring
$Orth(P'')=(G',S',X_0',X', Y',f')$ of $G'$, such that 
\begin{itemize}
\item $S(P'')=S'$;
\item  if $P''$ is smooth, then $Orth(P'')$ has a precoloring extension; and  
\item  if $c$ is a precoloring extension of $Orth(P'')$, then a precoloring 
extension of $P''$, and therefore of $P$, can be constructed from $c$ in 
polynomial time.
\end{itemize}
Moreover, $Orth(P'')$ can be constructed in polynomial time.

Let $\mathcal{L}=\{Orth(P'') \; : \; P'' \in \mathcal{L}_2\}.$
Now $\mathcal{L}$ has the following properties.
\begin{itemize}
\item $|\mathcal{L}| \leq |V(G)|^{q_1(q_2(|S|)}$;
\item $|S'| \leq q_1(q_2(|S|))$ for every $P' \in \mathcal{L}$; and
\item  if $c$ is a precoloring extension of $P' \in \mathcal{L}$, then a 
precoloring extension of  $P$ can be constructed from $c$ in polynomial time.
\item every $P' \in \mathcal{L}$ is orthogonal.
\end{itemize}
To complete the proof of the Theorem~\ref{orthogonalthm} we need to show that
if $P$ has a precoloring extension, then some $P' \in \mathcal{L}$ 
has a precoloring extension.
So assume that $P$ has a precoloring extension. Since $\mathcal{L}_1$ is equivalent to $P$, it follows that some $P_1 \in \mathcal{L}_1$ has a precoloring
extension. This implies that some 
$P_2 \in \mathcal{L}(P_1) \subseteq \mathcal{L}_2$ is smooth.
But now $Orth(P_2)$ has a precoloring extension, and 
$Orth(P_2) \in \mathcal{L}$. This completes the proof of 
Theorem~\ref{orthogonalthm}. \end{proof}

\section{Companion triples} \label{sec:companion}

In view of Theorem~\ref{orthogonalthm} we now focus on testing for the 
existence of a precoloring extension for an orthogonal excellent starred  
precoloring.

Let $G$ be a $P_6$-free graph, and let $P=\esspc$ be an orthogonal excellent starred  precoloring of $G$. We may assume that $X_0=X^0(P)$.
Let $\mathcal{C}(P)$ be the set of components of $G|Y^*$, and let 
$C \in \mathcal{C}(P)$.  It follows that $X \setminus X(C)$
is anticomplete to $V(C)$, and we may assume (using symmetry) that
$X(C) \subseteq X_{12} \cup X_{34}$. We now define the precoloring obtained from
$P$ by \emph{contracting the $ij$-neighbors of $C$}, or, in short,
by \emph{neighbor contraction}. We may assume that $\{i,j\}=\{1,2\}$.
Suppose that $X_{12} \cap X(C) \neq \emptyset$, and let
$x_{12} \in X_{12} \cap X(C)$. Let $\tilde{G}$ be the graph define as follows:
$$V(\tilde{G})=G \setminus (X_{12} \cap X(C)) \cup \{x_{12}\}$$ 
$$\tilde{G} \setminus \{x_{12}\}=G \setminus (X_{12} \cap X(C))$$
$$N_{\tilde{G}}(x_{12})=\bigcup_{x \in X_{12} \cap X(C)} N_G(x) \cap V(\tilde{G}).$$
Moreover, let
$$\tilde{X}=X \setminus (X_{12} \cap X(C)) \cup \{x_{12}\}.$$
Then $\tilde{P}=(\tilde{G},X_0,\tilde{X}, Y^*,f)$ is an orthogonal 
excellent  starred   precoloring of $\tilde{G}$.  We 
say that $\tilde{P}$ is \emph{obtained from $P$ by contracting the
$12$-neighbors of $C$}, or, in short, \emph{obtained from $P$ by neighbor
contraction}.
We call $x_{12}$ the image of $X_{12} \cap X(C)$, and
define $x_{12}(C)=x_{12}$. Observe that $x_{12} \in X$
(this fact simplifies notation later), and  that $M_P(v)=M_{\tilde{P}}(v)$
for every $v \in V(\tilde{G})$.

For $i,j \in \{1,2,3,4\}$ and $t \in X_0 \cup S$ let
$\tilde{G}_{ij}(t)=\tilde{G}|(\tilde{X}_{ij} \cup Y^* \cup \{t\})$. 
While graph  $\tilde{G}$ may not be $P_6$-free, the following weaker statement 
holds:

\begin{lemma}
\label{P6free}
Let $P$ be an excellent orthogonal precoloring of a $P_6$-free graph $G$.
Let $C \in \mathcal{C}(P)$ and assume that
$X(C) \cap X_{12}$ is non-empty. Let 
$\tilde{P}=(\tilde{G}, X_0,\tilde{X},Y^*, f)$ be obtained from $P$ by 
contracting the $12$-neighbors of $C$. Then $\tilde{G}_{ij}(t)$ is 
$P_6$-free for every $i,j \in \{1,2,3,4\}$
and $t \in S \cup X_0$.
\end{lemma}

\begin{proof} 
If $\{i,j\} \neq \{1,2\}$, then  $\tilde{G}_{ij}(t)$ is an 
induced subgraph of $G$, and therefore it is $P_6$-free. So we may assume that
$\{i,j\}=\{1,2\}$. 
Suppose that $Q=q_1- \ldots -q_6$ is a $P_6$ in $\tilde{G}_{ij}(t)$. Since
$\tilde{G}_{ij}(t) \setminus x_{12}$ is an induced subgraph of $G$, it follows 
that $x_{12} \in V(Q)$. If the neighbors of $x_{12}$ in $Q$ have a common neighbor $n \in X(C) \cap X_{12}$, then $G|((V(Q) \setminus \{x_{12}\}) \cup \{n\}) $ is a $P_6$ in $G$, a contradiction. It follows that  $x_{12}$ has two neighbors in $Q$, say $a,b$, each of $a,b$ has a neighbor in $X_{12} \cap X(C)$, and no 
vertex of $X(C) \cap X_{12}$ is complete
to $\{a,b\}$. Since $V(C)$ is complete to $X(C)$, it follows that
$a,b \not \in V(C)$, and so 
$a,b \in (X_{12} \setminus X(C)) \cup (Y^* \setminus V(C)) \cup \{t\}$.
Let $Q'$ be a shortest path from $a$ to $b$ with 
${Q'}^* \subseteq X(C) \cup V(C)$. Since $V(Q) \setminus \{a,b,t\}$ is 
anticomplete to $V(C)$, and $V(Q) \setminus \{a,b\}$ is anticomplete
to $X(C) \cap X_{12}$, it follows that $V(Q')$ is anticomplete to 
$V(Q) \setminus (\{x_{12}\} \cup \{a,b,t\})$.  Moreover, if $t \neq a,b$, then
$t$ is anticomplete to ${Q'}^* \setminus V(C)$. 
If follows that if 
$t \not  \in V(Q) \setminus \{a,b,x_{12}\}$ or $t$ is anticomplete to 
$V(Q') \cap V(C)$
then  $q_1-Q-a-Q'-b-Q-q_6$ is a path of length at least six in $G$, 
a contradiction; 
so $t \in V(Q) \setminus \{a,b,x_{12}\}$, and $t$ has a neighbor in 
$V(Q') \cap V(C)$.  Since $V(C)$ is complete to
$X(C)$, it follows that $|V(C) \cap V(Q')| = 1$, and  
$|{Q'}^*|=3$. Let $V(Q') \cap V(C)=\{q'\}$.
We may assume that $b$ has a neighbor $c \in V(Q) \setminus \{x_{12}\}$,
and if $a=q_i$ and $b=q_j$, then $i<j$. 
Since $a-Q'-b-c$ is not a $P_6$ in $G$, it follows that $t=c$.
But now $q_1-a-Q'-q-t-Q-q_6$ is a $P_6$ in $G$, a contradiction.
This proves Lemma~\ref{P6free}.
\end{proof}

Let $P=\esspc$ be an orthogonal excellent starred precoloring.
Let $H$ be a graph, and let
$L$ be a $4$-list assignment for $H$.  Recall that $X^0(L)$ is the set of 
vertices of $H$ with $|L(x_0)|=1$. Let $M$ be the list assignment
obtained from $M_P$ by updating $Y^*$ from $X_0$.
 We say that $(H,L,h)$ is a 
\emph{near-companion triple for $P$ with correspondence $h$} if 
there is an orthogonal excellent starred  precoloring 
$\tilde{P}=(\tilde{G},S,X_0,\tilde{X},Y^*,f)$ 
obtained from $P$ by a sequence of neighbor contractions, and
the following hold:
\begin{itemize}
\item $V(H) = \tilde{X} \cup Z$; 
\item $h: Z \rightarrow \mathcal{C}(P)$;
\item for every $z \in Z$, $N(z)=\tilde{X}(V(h(z)))$;
\item $H|(Z \cup \tilde{X}_{ij})$ is $P_6$-free for all $i,j$;
\item $Z$ is a stable set;
\item for every $x \in \tilde{X}$, $L(x) \subseteq M_P(x)=M(x)$;
\item for every $z \in Z$ such that  $L(z) \neq \emptyset$, if $q \in \{1,2,3,4\}$ and $q \not \in L(z)$, then some 
vertex $V(h(z))$ has a neighbor $u \in S \cup X_0 \cup X^0(L)$ with $f(u)=q$; and
\item for every $z \in Z$ and every $q \in L(z)$, there is
  $v \in V(h(z))$ with $q \in M(v)$, and no vertex $u \in S \cup X_0$
  with $f(u)=q$ is complete to $V(h(z))$.
\end{itemize}
For $z \in Z$, we call $h(z)$ the \emph{image of $z$}. 

If $(H,L,h)$ is a near-companion triple for $P$, and in addition
\begin{itemize}
\item $\tilde{P}$ has a precoloring extension if and only if $(H,L)$ is 
colorable, and a coloring of $(H,L)$ can be converted to a precoloring
extension of $P$ in polynomial time.
\end{itemize}
we say that $(H,L,h)$ is a \emph{companion triple} for $P$.

For $i,j \in \{1,2,3,4\}$ and $t \in S \cup X_0$ let $H_{ij}(t)$ be the graph 
obtained from $H|(\tilde{X}_{ij} \cup Z)$ by adding the vertex $t$ and 
making $t$ adjacent to 
the vertices of  $N_{\tilde{G}}(t) \cap \tilde{X}_{ij}$.
The following is a key  property of near-companion triples.

\begin{lemma}
\label{keycompanion}
Let $G$ be a $P_6$-free graph, let $P=\esspc$ be an orthogonal excellent 
starred precoloring of $G$, and let $(H,L,h)$ be a near-companion triple
for $P$. 
Let $M$ be the list assignment obtained from $M_P$ by updating $Y^*$ from 
$X_0$.
Assume that $L(v) \neq \emptyset$ for every $v \in V(H)$.
Let $i,j \in\{1,2,3,4\}$ and $t \in X_0 \cup S$, and let $Q$ be a 
$P_6$ in $H_{ij}(t)$.
Then $t \in V(Q)$, and there exists  $q \in V(Q) \setminus N(t)$ such that 
$f(t) \not \in M(q)$.
\end{lemma}

\begin{proof}
Since $H|(\tilde{X}_{ij} \cup Z)$ is $P_6$-free, it follows that $t \in V(Q)$. 
Suppose that for every $q \in V(Q) \setminus N(t)$, $f(t) \in L(q)$.
Let $z \in V(Q) \cap Z$.
Since $t$ is anticomplete to $Z$, it follows that $f(t) \in L(z)$ 
By  the definition of a near-companion triple, there is a vertex 
$q(z) \in V(h(z))$ such that $f(t) \in M(q(z))$. Since $M$ is obtained
from $M_P$ by updating $Y^*$ from $X_0$, it follows that
$t$ is non-adjacent to $q(z)$. 
Now replacing $z$ with $q(z)$ for every $z \in V(Q) \cap Z$, 
we get a $P_6$ in $\tilde{G}_{ij}(t)$ that contradicts Lemma~\ref{P6free}.
This proves Lemma~\ref{keycompanion}.
\end{proof}

The following is the main result of this section.

\begin{theorem}
\label{companion}
Let $G$ be a $P_6$-free graph and let $P=\esspc$ be an orthogonal excellent 
starred precoloring of $G$. Then there is a polynomial time algorithm that 
outputs a companion triple for $P$.
\end{theorem}

\begin{proof}
We may assume that $X_0=X^0(P)$. Let $M$ be the list assignment obtained from $M_P$ by updating $Y^*$ from $X_0$. Write $\mathcal{C}=\mathcal{C}(P)$.
For $Q \subseteq \{1,2,3,4\}$ and $C \in \mathcal{C}$, we say that a 
coloring $c$ of $(C,M)$ is a $Q$-coloring if $c(v) \in Q$ for every
$v \in V(C)$. 
Given  $Q \subseteq \{1,2,3,4\}$, we say that $Q$ is \emph{good for $C$}
if $(C,M)$ admits a proper $Q$-coloring, and \emph{bad for $C$} otherwise.
By Theorem~\ref{3colP7},  for every $Q$ with $|Q| \leq 3$, we can  
test in polynomial time if $Q$ is good for $C$.
Let $\mathcal{Q}(C)$ be the set of all inclusion-wise maximal 
bad subsets of $\{1,2,3,4\}$.  Observe that if $Q$ is bad, then all its 
subsets are bad.

Here is another useful property of $\mathcal{Q}(C)$.

\vspace*{-0.4cm}
  \begin{equation}\vspace*{-0.4cm} \label{recolor}
    \longbox{\emph{Let $Q \in \mathcal{Q}(C)$, and let $i \in Q$
be such that no $u \in S \cup X_0$ with $f(u)=i$ has a neighbor in $V(C)$. Then for every $j \in \{1,2,3,4\} \setminus  Q$, we have $(Q \setminus \{i\}) \cup \{j\} \in \mathcal{Q}(C)$.}}
\end{equation}

Suppose not. Let $Q'=Q \setminus \{i\} \cup \{j\}$. Let $c$ be a proper 
$Q'$-coloring
of $(C,M)$. It follows from the definition of $M$ that $i \in M(y)$
for every $y \in V(C)$.
Recolor every vertex $u \in V(C)$ with $c(u)=j$  with color $i$.
This gives a proper $Q$-coloring of $(C,M)$, a contradiction. This proves 
(\ref{recolor}).

\bigskip

First we describe a sequence of neighbor contractions to produce $\tilde{P}$ as
in the definition of a companion triple.
Let $C \in \mathcal{C}$ with $|V(C)|>1$.
Let $\{i,j,k,l\}=\{1,2,3,4\}$ and let 
$X(C) \subseteq X_{ij} \cup X_{kl}$.
We may assume (without loss of generality) that 
$X(C)  \subseteq X_{12} \cup X_{34}$. If $X(C)$
meets both $X_{12}$ and $X_{34}$,
contract the $1,2$-neighbors of $C$, and the $3,4$-neighbors of $C$;
observe that in this case $\tilde{X}(C)=\{x_{12}(C), x_{34}(C)\}$.
If $X(C)$ meets exactly one of $X_{12},X_{34}$, say $X(C) \subseteq X_{12}$,
and $\{3,4\}$ is bad for $C$, contract the $12$-neighbors of $C$.
Repeat this for every $Q \in \mathcal{Q}(C)$; let 
$\tilde{P}=(\tilde{G}, S, X_0, \tilde{X}, Y^*, f)$ be the
resulting precoloring. Observe that $\tilde{X} \subseteq X$.

\vspace*{-0.4cm}
  \begin{equation}\vspace*{-0.4cm} \label{tilde}
    \longbox{\emph{$P$ has a precoloring extension if and only if $\tilde{P}$ has a precoloring extension, and a precoloring extension of $\tilde{P}$
can be converted into a precoloring extension of $P$ in polynomial time. }}
\end{equation}

Since $|\mathcal{C}(P)| \leq |V(G)|$,
it is enough to show that the property of having a precoloring extension,
and the algorithmic property,
do not change when we perform  one step of the construction above.

Let us say that we start with $P_1=(G_1,S, X_0, X_1, Y^*, f)$ and
finish with $P_1=(G_2, S,X_0, X_2, Y^*, f)$.
We claim that in all cases, each of  the sets that is being contracted
(that is, replaced by its image) 
is monochromatic in every precoloring  extension of $P$. 

Let $C \in \mathcal{C}(P)$ with $|V(C)|>1$, such that $P_2$ is obtained
from $P_1$ by contracting neighbors of $C$.
Let $\{i,j,k,l\}=\{1,2,3,4\}$ and let 
$X_1(C) \subseteq X_{ij} \cup X_{kl}$. 
If $\tilde{X}(C)$ meets both $X_{ij}$ and $X_{kl}$, then since
$|V(C)|>1$,  each 
of the sets $X_1(C) \cap X_{ij}$, $X_1(C) \cap X_{kl}$
is monochromatic in every precoloring extension of $P_1$, as required.
So we may assume  that $X_1(C) \subseteq X_{ij}$.
Now $X_1(C)$ is monochromatic in every precoloring extension of $P_1$
because the set $\{k,l\}$ is bad for $C$. This proves the claim.

Now suppose that a set $A$ was contracted to produce its image $a$.
If $P_1$ has a precoloring extension,  we can give $a$ the unique color that 
appears  in $A$, thus constructing an extension of $P_2$. On the other hand, 
if $P_2$ has a precoloring extension, then every vertex of $A$ can be colored 
with the color of $a$. This proves (\ref{tilde}).

\bigskip

Next we define $L: \tilde{X} \rightarrow 2^{[4]}$.
Start with $L(x)=M_{\tilde{P}}(x)$ for every $x \in \tilde{X}$.
Again let $C \in \mathcal{C}$ with $|V(C)|>1$, let 
$\{i,j,k,l\}=\{1,2,3,4\}$, and let
$\tilde{X}(C) \subseteq X_{ij} \cup X_{kl}$.
For every  $Q \in \mathcal{Q}(C)$ such that 
$Q=\{1,2,3,4\} \setminus \{i\}$, update $L$ by removing $i$ from
$L(x)$ for every $x \in X_{ij} \cap \tilde{X}(C)$.

Next assume that $X(C)$ meets both $X_{ij},X_{kl}$ ,
the sets $\{i,k\},\{i,l\}$ are good for $C$,
and the sets $\{j,k\},\{j,l\}$ are bad for $C$.
Update $L$ by removing $i$ from $L(x_{ij}(C))$.

Finally, assume that $X(C)$ meets both $X_{ij},X_{kl}$ ,
the set $\{i,k\}$ is good for $C$,
and the sets $\{i,l\},\{j,k\},\{j,l\}$ are bad for $C$.
Update $L$ by removing $i$ from $L(x_{ij}(C))$
and
by removing $k$ from $L(x_{kl}(C))$.

Now the following holds.

\vspace*{-0.4cm}
  \begin{equation}\vspace*{-0.4cm} \label{X(C)}
    \longbox{\emph{Let $\{1,2,3,4\}=\{i,j,k,l\}$ and let
$C \in \mathcal{C}$ such that $X(C) \subseteq X_{ij} \cup X_{kl}$.
\begin{enumerate}
\item If $\{1,2,3,4\} \setminus \{i\} \in \mathcal{Q}(C)$, then
$i \not \in \bigcup_{x \in \tilde{X}(C)}L(x)$.
\item 
If $\tilde{X}(C)$ meets both $X_{ij}$ and $X_{kl}$ and
$\{i,k\}, \{i,l\}$ are both good for $C$, and
$\{j,k\},\{j,l\}$ are both bad for $C$, then 
$i \not \in  L(x_{ij}(C)) \cup L(x_{kl}(C))$.
\item  If $\tilde{X}(C)$ meets both $X_{ij}$ and $X_{kl}$ and
$\{i,k\}$ is  good for $C$, and
$\{i,l\}, \{j,k\},\{j,l\}$ are  bad for $C$, then 
$i,k \not \in  L(x_{ij}(C)) \cup L(x_{kl}(C))$.
\end{enumerate}
}}
\end{equation}

Next we  show that:

\vspace*{-0.4cm}
  \begin{equation}\vspace*{-0.4cm} \label{lists}
    \longbox{\emph{If $c$ is a precoloring extension of $\tilde{P}$, then 
$c(x) \in L(x)$ for every $x \in \tilde{X}$.}}
\end{equation}

This is clear for $x$ such that $L(x)=M(x)$, so let $x \in \tilde{X}$ be
such that $L(x) \neq M(x)$. Then
there exists $C \in \mathcal{C}$ with $|V(C)|>1$, and
$\{i,j,k,l\}=\{1,2,3,4\}$ with
$\tilde{X}(C) \subseteq X_{ij} \cup X_{kl}$, 
such that $x \in \tilde{X}(C)$.
Suppose that $c(x) \in M(x) \setminus L(x)$.  
Observe that $c|V(C)$ is a coloring of $(C,M)$. 
There are three possible situations in which $c(x)$ could have been removed 
from 
$M(x)$ to produce $L(x)$.

\begin{itemize}
\item $\{1,2,3,4\} \setminus \{i\}$ is bad for $C$, and $x \in X_{ij}$, 
and $c(x)=i$. 
In this case, since $(C,M)$ is not $\{1,2,3,4\} \setminus \{i\}$-colorable, it 
follows that some $v \in V(C)$ has $c(v)=i$, but
$V(C)$ is complete to $\tilde{X}(C)$, a contradiction.
\item  $\tilde{X}(C)$ meets both $X_{ij}$ and $X_{kl}$, the 
sets $\{i,k\},\{i,l\}$ are good for $C$,
the sets $\{j,k\},\{j,l\}$ are bad for $C$, 
$x=x_{ij}(C)$, and $c(x)=i$.
Since $\tilde{X}(C) \cap X_{kl} \neq \emptyset$, it follows that
$c(u) \in \{k,l\}$ for some $u \in \tilde{X}(C)$. Since the sets
$\{j,k\},\{j,l\}$ are bad for $C$ and $|V(C)|>1$, it follows that
$c(v)=i$ for some $v \in V(C)$, but $x_{ij}(C)$ is 
complete to $V(C)$, a contradiction.
\item $\tilde{X}(C)$ meets both $X_{ij}$ and $X_{kl}$, the 
set $\{i,k\}$ is good for $C$,
the sets $\{i,l\}, \{j,k\},\{j,l\}$ are bad for $C$, and
either $x = x_{ij}(C)$ and $c(x)=i$, or $x=x_{kl}(C)$ and $c(x)=l$.
Since $\tilde{X}(C)$ meets both $X_{ij}$ and $X_{ik}$ and $|V(C)|>1$, it 
follows that
$|c(V(C)) \cap \{i,j\}| =1$, and $|c(V(C)) \cap \{k,l\}| =1$.
Since  $\{j,k\},\{j,l\}$ are bad for $C$, it follows that
for some $v \in V(C)$ has $v(c)=i$, and so $c(x_{ij}(C)) \neq i$.
Since $\{i,l\}$ is bad for $C$, it follows that $c(V(C))=\{i,k\}$,
and so $c(x) \neq k$, in both cases a contradiction.
\end{itemize}
This proves (\ref{lists}).

\bigskip

Finally, for every $C \in \mathcal{C}$,  we   construct  the set $h^{-1}(C)$ 
and define $L(v)$ for every $v \in h^{-1}(C)$.

If $|V(C)|=1$, say $C=\{y\}$, let $h^{-1}(C)=\{y\}$, and let
$L(y) = M(y)$. 

Now assume $|V(C)|>1$. 
We may assume that  $\tilde{X}(C) \subseteq X_{12} \cup X_{34}$.

If all subsets of $\{1,2,3,4\}$ of size three are bad, then
set $h^{-1}{C}=\{z\}$ and $L(z)=\emptyset$.
From now on we assume that there is a good subset for $C$  of size at most 
three.

If  $\tilde{X}(C) \subseteq X_{12}$ or 
$\tilde{X}(C) \subseteq X_{34}$, set  $h^{-1}(C)=\emptyset$. 

So we may assume that $\tilde{X}(C)$ meets both $X_{12}$ and $X_{34}$.
If all sets of size two, except $\{1,2\}$ and $\{3,4\}$ are bad for $C$, set 
$h^{-1}{C}=\{z\}$ and $L(z)=\emptyset$.
Next let $Q \in \mathcal{Q}(C)$ with $|Q|=2$; write  $\{i,j,k,l\}=\{1,2,3,4\}$,
and say $Q=\{i,j\}$. 
We say that $Q$ is \emph{friendly} if
there exist $u_i,u_j \in S \cup X_0$,
both with neighbors in $C$, and with $f(u_i)=i$ and $f(u_j)=j$.
For every friendly set $Q$, let $v(C,Q)$ be  a new vertex, and let
$h^{-1}(C)$ consist of all such vertices $v(C,Q)$. 
Set $L(v(C,Q))=\{1,2,3,4\} \setminus Q$.

Let $Z=\bigcup_{C \in \mathcal{C}}h^{-1}(C)$.
Finally, define the correspondence function $h$ by setting 
$h(z)=C$ for every $z \in h^{-1}(C)$ and $C \in \mathcal{C}$.

Now we define $H$. We set $V(H)=\tilde{X} \cup Z$, and
$pq \in E(H)$ if and only if either
\begin{itemize}
\item $p,q \in \tilde{X}$ and $pq \in E(G)$, or
\item there exists $C \in \mathcal{C}$ such that $p \in h^{-1}(C)$ and
$q \in \tilde{X}(C)$.
\end{itemize}

The triple $(H,L,h)$ that we have constructed satisfies the following.
\begin{itemize}
\item $\tilde{X} \subseteq V(H)$; write $Z=V(H) \setminus \tilde{X}$.
\item $N(z)=\tilde{X}(V(h(x)))$ for every $z \in Z$. 
\item $Z$ is a stable set.
\item For every $x \in \tilde{X}$, $L(x) \subseteq M_P(x)=M(x)$.
\item $h: Z \rightarrow \mathcal{C}(P)$.
\item If $z \in Z$ with $L(z) \neq \emptyset$, and $q \in \{1,2,3,4\} \setminus L(z)$, then some 
vertex $V(h(z))$ has a neighbor $u \in S \cup X_0$ with $f(u)=q$.
(This is in fact stronger than what is required in the definition of a companion triple; we will relax this condition later.)
\end{itemize}

To complete the proof of Theorem~\ref{companion}, it remains to show the
following
\begin{enumerate}
\item For every $z \in Z$ and every $q \in L(z)$,  
there is $v \in V(h(z))$ with $q \in M(v)$, and 
no vertex $u \in S \cup X_0$ with $f(u)=q$ is complete to $V(h(z))$.
\item for every $i,j \in \{1,2,3,4\}$, $H|(\tilde{X_{ij}} \cup Z)$ is $P_6$-free.
\item  $P$ has a precoloring extension if and only if $(H,L)$ is 
colorable, and a proper coloring of $(H,L)$ can be converted to a precoloring
extension of $P$ in polynomial time.
\end{enumerate}

We prove the first statement first.
Let $z \in Z$ and $q \in L(z)$, and suppose that for every $v \in V(h(z))$
$q \not \in M(v)$, or some vertex $u \in S \cup X_0$ with $f(u)=q$ is 
complete to $V(h(z))$.
It follows that  $|V(h(Z))|>1$.
Since $z \in Z$, it follows that there exists
a set $\{i,j\} \in \mathcal{Q}(h(Z))$ and 
$L(z)=\{1,2,3,4\} \setminus \{i,j\}$. But now it follows that
$\{q,i,j\}$ is also bad for $h(Z)$, contrary to the maximality of $\{i,j\}$.
This proves the first statement.

Next we prove the second statement. 
By Lemma~\ref{P6free}, $\tilde{G}|(\tilde{X_{ij}} \cup Y^*)$ is $P_6$-free for
every $i,j \in \{1,2,3,4\}$. Suppose $Q$ is a $P_6$ in $H$. 
Let $C \in \mathcal{C}(P)$. Since  no vertex of 
$V(H) \setminus h^{-1}(C)$ is mixed on $h^{-1}(C)$, it follows that
$|V(Q) \cap h^{-1}(C)| \leq 1$. Moreover, 
$\tilde{X}_{ij}(h^{-1}(C))= \tilde{X}_{ij}(C)$. Let $G'$ be obtained
from $\tilde{G}$ by replacing each $C \in \mathcal{C}$ by a single vertex
of $C$, choosing this vertex to be in $V(Q)$ if possible.  Then $G'$ is an 
induced subgraph of $G$, and $Q$ is a $P_6$ in $G'$, a contradiction.
This proves the second statement.

Finally we prove the last statement.
Let
$\mathcal{C}_1=\{C \in \mathcal{C} \; : \;  |V(C)|=1\}$, and
let $Y=\bigcup_{C \in \mathcal{C}_1}V(C)$. Then $Y \subseteq Z$.

Suppose first that $P$ has a precoloring extension. By~(\ref{tilde}),
there exists a precoloring extension of $\tilde{P}$; denote it by $c$.
By~(\ref{lists}), $c|(\tilde{X} \cup Y)$ is a coloring of $(H|(X \cup Y),L)$.
It remains to show that $c$ can be extended to $Z \setminus Y$. 
Let $z \in Z$, and let $h(z)=C$. Then there is a friendly 
set $\{i,j\} \in \mathcal{Q}$ such that $z=v(C,Q)$. Since 
$Z$ is a stable set, in order to show that $c$ can be extended to 
$Z \setminus Y$, it
is enough to show that 
$$L(z) \not \subseteq c(\tilde{X}(C)).$$
Since $L(v(C,Q))=\{1,2,3,4\} \setminus Q$, it is enough to show that 
$$\{1,2,3,4\} \setminus c(\tilde{X}(C)) \not \subseteq Q.$$
But the latter statement is true because 
$$ c(V(C)) \subseteq \{1,2,3,4\} \setminus c(\tilde{X}(C))$$
 and 
$c(V(C))$
is a good set, and therefore $c(V(C)) \not \subseteq Q$.
This proves that if $\tilde{P}$ has a precoloring extension, then $(H,L)$ is
colorable.

Now let $c$ be a proper coloring of $(H,L)$. By (\ref{tilde}) it is enough to 
show that $\tilde{P}$ has a precoloring extension.
We define a precoloring extension 
$\tilde{c}$ of $\tilde{P}$. Set $\tilde{c}(v)=f(v)$ for every 
$v \in S \cup X_0$,
and $\tilde{c}(x)=c(x)$ for every $x \in \tilde{X} \cup Y$. It follows from 
the definition of $L$ that 
$\tilde{c}$ is a precoloring extension of 
$(\tilde{G} \setminus (Y^* \setminus Y), S, X_0, \tilde{X},Y)$.

Let $C \in \mathcal{C}$ with $|V(C)|>2$. We extend $\tilde{c}$ to $C$.
We will show that  for every $Q \in \mathcal{Q}(C)$,
$\{1,2,3,4\}  \setminus c(\tilde{X}(C)) \not \subseteq Q$.
Consequently $T=\{1,2,3,4\}  \setminus c(\tilde{X}(C))$ is good for $C$.
Since some vertex of $S \cup X_0 \cup \tilde{X}$ is complete to $V(C)$,
it follows that $|T| \leq 3$. Therefore 
we can define 
$\tilde{c} : V(C) \rightarrow \{1,2,3,4\}$ to be a proper $T$-coloring of $(C,M)$,
which can be done in polynomial time by Theorem~\ref{3colP7}.

So suppose that there is $Q \in \mathcal{Q}(C)$ such that
$\{1,2,3,4\} \setminus c(\tilde{X}(C)) \subseteq Q$.
Then $\{1,2,3,4\} \setminus Q  \subseteq c(\tilde{X}(C))$.
By (\ref{X(C)}.1), $|Q|<3$. 

We may assume that $\tilde{X}(C) \subseteq X_{12} \cup X_{34}$.
Suppose first that $\tilde{X}(C)$ meets both $X_{12}$ and $X_{34}$,
and so $\tilde{X}(C)=\{x_{12}(C),x_{34}(C)\}$. Then $|c(\tilde{X}(C))|=2$,
and so $|Q| \neq 1$. Therefore  may assume that
$|Q|=2$. If $Q$ is friendly, then $c(v(C,Q)) \not \in Q$, 
and so $\{1,2,3,4\} \setminus Q  \not \subseteq  c(\tilde{X}(C))$, so
we may assume that $Q$ is not friendly. By symmetry, we may assume that
$Q \in\{\{1,2\},\{1,3\}\}$. If $Q=\{1,2\}$, then since
$L(x_{12}(C)) \subseteq \{1,2\}$, it follows that
$\{1,2,3,4\} \setminus Q  \not \subseteq  c(\tilde{X}(C))$, so we may 
assume that $Q=\{1,3\}$.

Suppose first that for every $i \in Q$, there is no vertex $u \in S \cup X_0$
with $c(u) =i$ and such that $u$ has a neighbor in $V(C)$. Now (\ref{recolor})
implies that every set of size two is bad for $C$. 
Therefore  $h^{-1}(C)=\{z\}$ and $L(z)=\emptyset$, contrary to the fact that
$c$ is a proper coloring of $(H,L)$. 

We may assume from symmetry that
\begin{itemize}
\item there is a vertex $u \in S \cup X_0$
with $c(u) =1$ and such that $u$ has a neighbor in $V(C)$.
\item there is no vertex $u \in S \cup X_0$
with $c(u) =3$ and such that $u$ has a neighbor in $V(C)$.
\end{itemize}

Now by (\ref{recolor})   all the
sets sets $\{1,2\},\{1,3\},\{1,4\}$ are bad.
If the only good set is $\{3,4\}$, then $L(z)=\emptyset$, 
contrary to the fact that $c$ is a coloring of $(H,L)$.
Therefore,  at least one of $\{2,3\},\{2,4\}$ is good,  and (\ref{X(C)}.2)
and (\ref{X(C)}.3) imply that 
$2 \not \in  L(u)$ for every $u \in \tilde{X}(C)$,  contrary to the fact that
$2 \in \{1,2,3,4\} \setminus Q  \subseteq c(\tilde{X})$.
This proves that not both $\tilde{X}(C) \cap X_{ij}$ and
$\tilde{X}(C) \cap X_{kl}$ are non-empty.

We may assume that $\tilde{X}(C) \subseteq X_{12}$. Then 
$c(\tilde{X}(C)) \subseteq \{1,2\}$, and so $3,4 \in Q$. Since $|Q|<3$, 
we have $Q=\{3,4\}$. 
It follows  from the construction of $\tilde{G}$ that $|\tilde{X}(C)| \leq 1$, 
contrary to the fact that
$\{1,2,3,4\} \setminus Q  \subseteq \bigcup_{u \in X(C)}\{c(u)\}$.
This completes the proof of the second statement, and 
Theorem~\ref{companion} follows.
\end{proof}

\section{Insulating cutsets} \label{insulatingcutsets}

Our next goal is to transform companion triples further, restricting them in 
such a way that we can test colorability.

Let $H$ be a graph and let $L$ be a $4$-list assignment for $H$.  We
say that $D \subseteq V(H)$ is a \emph{chromatic cutset} in $H$ if
$V(H)=A \cup B \cup D$, $A \neq \emptyset$, and $a \in A$ is adjacent
to $b \in B$ only if $L(a) \cap L(b)=\emptyset$. For
$i,j \in \{1,2,3,4\}$ let
$D_{ij}=\{d \in D \; : \; L(d) \subseteq \{i,j\}\}$. The set $A$ is
called the \emph{far side} of the chromatic cutset.  We say that a
chromatic cutset $D$ is \emph{$12$-insulating} if
$D=D_{12} \cup D_{34}$ and for every
$\{p,q\} \in \{\{1,2\}, \{3,4\}\}$ and every component $\tilde{D}$ of
$H|D_{pq}$ the following conditions hold.
\begin{itemize}
\item $\tilde{D}$ is bipartite; let $(D_1,D_2)$ be the bipartition.
\item $|L(d)|=|L(d')|$ for every $d,d' \in D_1 \cup D_2$. 
\item There exists $a \in A$ with  a neighbor in $\tilde{D}$ and with
$L(a) \cap \{p,q\} \neq \emptyset$.
\item Suppose that $|L(d)|=2$ for every $d \in V(\tilde{D})$.
Write $\{i,j\}=\{p,q\}$ and let $\{s,t\}=\{1,2\}$.
If $a \in A$ has a neighbor in $d \in D_s$ and $i \in L(a)$, and 
$b \in B$ has a neighbor in $\tilde{D}$, then 
\begin{itemize}
\item if $b$ has a neighbor in $D_s$, then $j \not \in L(b)$, and
\item if $b$ has a neighbor in $D_t$, then $i \not \in L(b)$.
\end{itemize}
\end{itemize}
Insulating cutsets are useful for the following reason.
We say that a component $\tilde{D}$ of $H|D_{pq}$ is \emph{complex} if 
$|L(d)|=2$ for every $d \in V(\tilde{D})$.

\begin{theorem}
\label{insulating}
Let $D$ be a $12$-insulating chromatic cutset in $(H,L)$, and let $A,B$ be as 
in the  definition of an insulating cutset. Let $D'$ be the union of the vertex 
sets of complex components of  $H|D_{12}$ and of $H|D_{34}$, and let 
$D''=D \setminus D'$.
If $(H|(B \cup D''),L)$ and $(H \setminus B, L)$
are both colorable, then $(H,L)$ is colorable.
Moreover, given proper colorings of $(H|(B \cup D''),L)$ and 
$(H \setminus B, L)$, a proper coloring of $(H,L)$ can be found in polynomial time.
\end{theorem} 

\begin{proof}
Let $c_1$ be a proper coloring of $(H|(B \cup D''),L)$ and let $c_2$ be a 
proper coloring of  $(H \setminus B, L)$.

A \emph{conflict} in $c_1,c_2$ is a pair of adjacent vertices $u,v$ such that
$c_1(u)=c_2(v)$. Since $c_1,c_2$ are both proper colorings and $D$ is a 
chromatic cutset,
and $|L(d)|=1$ for every $d \in D''$, we deduce that 
every conflict involves one vertex of $D'$ and one vertex of $B$.
Below we describe a polynomial-time procedure that modifies $c_2$ to reduce 
the number of conflicts (with $c_1$ fixed).

Let $u \in D'$ and $v \in B$ be a conflict.
Then $uv \in E(H)$ and $c_1(u)=c_2(v)$.
Let $\tilde{D}$ be the component of
$G|D$ containing $u$. Then $V(\tilde{D}) \subseteq D'$ and
$\tilde{D}$ is bipartite; let $(D_1,D_2)$ be the 
bipartition of $\tilde{D}$. We may assume that $u \in D_1$. We may also
assume that $L(d)=\{1,2\}$ for every $d \in V(\tilde{D})$, and that 
$c_1(u)=c_2(v)=2$. Since $L(d)=\{1,2\}$ for every $d \in V(\tilde{D})$, it 
follows that for every $i \in \{1,2\}$ and $d \in D_i$, we have $c_2(d)=i$. 
Let $c_3$ be obtained from $c_2$ by setting
$c_3(d)=1$ for every $d \in D_2$; $c_3(d)=2$ for every $d \in D_1$; and
$c_3(d)=c_2(d)$ for every $w \in (A \cup D) \setminus (D_1 \cup D_2)$.
(This modification can be done in linear time).

First we show that $c_3$ is a proper coloring of $(H \setminus B, L)$. 
Since $L(d)=\{1,2\}$ for every $d \in V(\tilde{D})$, 
$c_3(v) \in L(v)$ for every $v \in A \cup D$.
Suppose there exist adjacent $xy \in D \cup A$ such that
$c_3(x)=c_3(y)$. Since $\tilde{D}$ is a component of $H|D$, 
we may assume that $x \in D_1 \cup D_2$ and $y \in A$. 
Suppose first that $x \in D_1$. Then $c_3(y)=c_3(x)=2$,
and so $2 \in L(y)$  and $y$ has a neighbor in $D_1$. But  
$v \in B$ has a  neighbor in $D_1$ and $1 \in L(v)$, which is a contradiction.
Thus we may assume that $x \in D_2$. Then $c_3(y)=c_3(x)=1$,
and so $1 \in L(y)$ and $y$ has a neighbor in $D_2$. But $v \in B$
has a neighbor in $D_1$, and $1 \in L(b)$, again a contradiction. This proves
that $c_3$ is a proper coloring of $(H \setminus B, L)$.

Clearly $u,v$ is not a conflict in $c_1,c_3$. We show that no new conflict 
was created. Suppose that there is a new conflict, namely there exist
adjacent $u' \in D'$ and $v' \in B$ such that 
$c_1(v')=c_3(u')$, but $c_1(v') \neq c_2(u')$. 
Then $u' \in V(\tilde{D})$.
If $u' \in D_1$, then both $v$ and $v'$ have neighbors in $D_1$,
and $1 \in L(v)$, and $2 \in L(v')$; if $u' \in D_2$, then 
$v$ has a neighbor in $D_1$ and $v'$ has a neighbor in $D_2$, 
and $1 \in L(v') \cap L(v)$; and in both cases we get a contradiction.
Thus the number of conflicts in $c_1,c_3$ was reduced.

Now applying this procedure at most $|V(G)|^2$ times we obtained a proper
coloring $c_1'$ of  $(H|(B \cup D''),L)$ and a proper coloring $c_2'$ of 
$(H \setminus B, L)$ such that there is no conflict in $c_1',c_2'$.
Now define $c(v)=c_1'(v)$ if $v \in B \cup D''$
and $c(v)=c_2'(v)$ if $v \in V(H) \setminus B$; then $c$ is a proper coloring 
of  $(H,L)$. This proves Theorem~\ref{insulating}. \end{proof}

Let $G$ be a $P_6$-free graph, let $P=\esspc$ be an orthogonal excellent 
starred    precoloring of $G$, and let $(H, L, h)$ be a companion triple 
for $P$. Let $\{i,j,k,l\}=\{1,2,3,4\}$. Let 
$Z^{ij}=\{z \in Z \; : \; N(z) \cap \tilde{X} \subset X_{ij} \cup X_{kl}\}$.
It follows from the definition of a companion triple that
$Z^{ij}=Z^{kl}$ and that  $Z=\bigcup_{i,j \in \{1,2,3,4\}}Z^{ij}$.
Next we prove a lemma that will allow us to replace a companion triple for
$P$ with a polynomially sized collection of near-companion triples for $P$,
each of which has a useful insulating cutset. We will apply this lemma 
several times, and so we need to be able to apply it to near-companion triples 
for $P$, as well as to companion triples.

\begin{lemma} 
\label{insulatinglemma} 
There is function $q: \mathbb{N} \rightarrow \mathbb{N}$ such that the
following holds.  Let $G$ be a $P_6$-free graph, let $P=\esspc$ be an
orthogonal excellent starred precoloring of $G$, and let $(H, L, h)$
be a near-companion triple for $P$. Then there is an algorithm with
running time $O(|V(G)|^{q(|S|)})$ that outputs a collection
$\mathcal{L}$ of $4$-list assignments for $H$ such that
\begin{itemize}
\item $|\mathcal{L}| \leq |V(G)|^{q(|S|)}$;
\item if $L' \in \mathcal{L}$ and $c$ is a proper coloring of $(H,L')$, then
$c$ is a proper coloring of $(H,L)$; and
\item if $(H,L)$ is colorable, then there exists $L' \in \mathcal{L}$ such that
$(H,L')$  is colorable.
\end{itemize}
Moreover, for every  $L' \in \mathcal{L}$,
\begin{itemize}
\item $L'(v) \subseteq L(v)$ for every $v \in V(H)$;
\item $(H,L',h)$ is a near companion triple for $P$; 
\item if for some $i,j \in \{1,2,3,4\}$ $(H,L)$ has an $ij$-insulating cutset 
$D'$ with far side $Z^{ij}$, then  $D'$ is an $ij$-insulating cutset with far side $Z^{ij}$
in $(H,L',h)$; and
\item $(H,L')$ has a $12$-insulating cutset $D \subseteq \tilde{X}$ 
with far side $Z^{12}$.
\end{itemize}
\end{lemma}

\begin{proof}
 Let $\tilde{P}=(\tilde{G}, S,X_0, \tilde{X}, Y^*, f)$ be as in the
definition of a near-companion triple. 
Assume that $Z^{12} \neq \emptyset$. If one of the graphs
$\tilde{G}|\tilde{X}_{12}$ and $\tilde{G}|\tilde{X}_{34}$ is not   bipartite,
set $\mathcal{L}=\emptyset$. From now on
we assume that $\tilde{G}|\tilde{X}_{12}$ and $\tilde{G}|\tilde{X}_{34}$ are   
bipartite.
We may assume that $X_0=X^0(\tilde{P})$.
Let $T_1, \ldots, T_p$ be types of $\tilde{X}$ with
$|L_P(T_i)|=2$ and such that $|L_P(T_i) \cap \{1,2\}|=1$. It follows that
$|L_P(T_i) \cap \{3,4\}|=1$.
Let $\mathcal{Q}$ be the set of all $2m$-tuples $Q=(Q_1, \ldots, Q_m,P_1, \ldots, P_m)$ such that
\begin{itemize}
\item $|Q_i| \leq 1$, $Q_i \subseteq \tilde{X}(T_i)$, and if $Q_i=\{q\}$,
then $L(q) \cap \{1,2\} \neq \emptyset$.
\item  $|P_i| \leq 1$, $P_i \subseteq \tilde{X}(T_i)$, and if $P_i=\{p\}$,
then $L(p) \cap \{3,4\} \neq \emptyset$.
\end{itemize}

For $x \in \tilde{X} \setminus (X_{12} \cup X_{34})$ and $z \in Z^{12}$ we say 
that $z$ is a \emph{$12$-grandchild} of
$x$ if there is a component $C$ of $\tilde{X}_{12}$ such that both 
$x$ and $z$ have neighbors in $V(C)$; a \emph{$34$-grandchild}  is 
defined similarly. Let $G_{12}(x)$ be the set of $12$-grandchildren of $x$;
define $G_{34}(x)$ similarly.

We define a $4$-list assignment $L'_Q$ for $H$.  Start with $L_Q'=L$.
For every $i \in \{1, \ldots, m\}$, proceed as follows.
If $|Q_i|=1$, say $Q_i=\{q_i\}$, set $L_Q'(q_i)$ to be the unique element of
$L(q_i) \cap \{1,2\}$. For every $x \in \tilde{X}(T_i)$ such that 
$G_{12}(q_i) \subset G(x)$ and $G_{12}(x) \setminus G_{12}(q_i) \neq \emptyset$,
update $L'_Q(x)$ by 
removing from it the  unique element of $L(x) \cap \{1,2\}$. 
Next assume that  $Q_i = \emptyset$. In this case, for every 
$x \in \tilde{X}(T_i) \setminus \{q_i,p_i\}$  
such that $x$ has a grandchild, update $L'_Q(x)$ by removing from it the 
unique element of $L(x) \cap \{1,2\}$. 

If $|P_i|=1$, say $P_i=\{p_i\}$, set $L_Q'(p_i)$ to be the unique element of
$L(p_i) \cap \{3,4\}$. 
For every $x \in \tilde{X}(T_i)$ such that 
$G_{34}(p_i) \subset G(x)$ and $G_{12}(x) \setminus G_{12}(p_i) \neq \emptyset$, 
update $L'_Q(x)$ by 
removing from it the  unique element of $L(x) \cap \{3,4\}$. 
Next assume that  $P_i = \emptyset$. In this case, for every $x \in X(T_i) \setminus \{p_i,q_i\}$ 
such that some component of $H|\tilde{X}_{34}$ contains both a neighbor of $x$ 
and a neighbor of a vertex in $Z^{12}$, update $L'_Q(x)$ by removing from it the 
unique element of $L(x) \cap \{3,4\}$.

If some vertex $z \in \tilde{X} \setminus \tilde{X}_{12}$ has 
neighbors  on both sides of the bipartition
of a component of $H|(\tilde{X}_{12})$, set 
$L'_Q(z)=L(z) \setminus \{1,2\}$.
If some vertex $z \in \tilde{X} \setminus \tilde{X}_{34}$ has neighbors 
on both sides of the bipartition
of a component of $H|(\tilde{X}_{34})$, set 
$L'_Q(z)=L(z) \setminus \{3,4\}$.
Finally, set $L'_Q(v)=L(v)$ for every other $v \in V(H)$ not yet specified.
Now let $L_Q$ be obtained from $L_Q'$ by updating exhaustively from 
$\bigcup_{i=1}^m (P_i \cup Q_i)$.

We need to check the following statements.
\begin{enumerate}
\item $L_Q(v) \subseteq L(v)$ for every $v \in V(H)$.
\item $(H,L_Q,h)$ is a near-companion triple of $P$.
\item If for some $i,j \in \{1,2,3,4\}$ $(H,L)$ has an $ij$-insulating cutset $D'$ with far side $Z^{ij}$, then  $D'$ is an $ij$-insulating cutset with far side $Z^{ij}$
in $(H,L_Q)$.
\item $(H,L_Q)$ has a $12$-insulating cutset with far side $Z^{12}$.
\end{enumerate} 

Clearly $L_Q(v) \subseteq L(v)$ for every $v \in V(H)$, and consequently
it is routine to check that the third statement holds, and that
in order to prove the  second statement it is sufficient to prove the following:

\vspace*{-0.4cm}
  \begin{equation}\vspace*{-0.4cm} \label{statement3}
    \longbox{\emph
{Set $f(x)=L_Q(x)$ for every  $x \in X^0(L_Q)$. 
Then for every $z \in Z$ with $L(z) \neq \emptyset$ and 
$q \in \{1,2,3,4\}$ such that  $q \not \in L_Q(z)$, there is
a vertex in $h(z)$ that has a neighbor $u \in S \cup X_0 \cup X^0(L_Q)$ with  
$f(u)=q$.}}
\end{equation}

We now prove this statement.
Let $z \in Z$ and 
$q \in \{1,2,3,4\}$ such that  $q \not \in L_Q(z)$.
We need to show that  there is
a vertex in $h(z)$ that has a neighbor $u \in S \cup X_0 \cup X^0(L')$ with 
$f(u)=q$. If $q \not \in L(z)$, the claim follows from the fact
that $(H,L,h)$ is a near-companion triple for $P$, so we may assume that 
$q \in L(z)$, and therefore
$z$ has a neighbor $u$ in $X^0(L_Q)$ with $f(u)=q$. Since  $Z$ is stable, it
follows that  $u \in \tilde{X}$, and therefore, by the definition
of a companion triple, $u$ is complete to $V(h(z))$. This proves (\ref{statement3}).

\bigskip

Finally, we prove that $(H,L_Q)$ has a $12$-insulating cutset with far side $Z^{12}$.
Let $D^1, \ldots, D^t$ be the components of
$H|\tilde{X}_{12}$
that contain a vertex $x$ such that $x$ has a neighbor 
$z$ in $Z^{12}$ with $L_Q(x) \cap L_Q(z) \neq \emptyset$.
Let $F^1, \ldots, F^s$ be defined similarly for $\tilde{X}_{34}$.
Let $D=X^0(L_Q) \cup \bigcup_{i=1}^t V(D_i) \cup \bigcup_{j=1}^w V(F_j)$.
We claim that $D$ is the required cutset.
Clearly $D$ is a
chromatic cutset, setting the far side to be $Z^{12}$ and $B=V(H) \setminus (A \cup D)$, and
the first two bullets of the definition of an insulating cutset are satisfied.
Let $\tilde{D} \in \{D_1, \ldots, D_t\}$ (the argument is symmetric for
$F_1, \ldots, F_s$).  
We need to check the following properties. 
\begin{itemize}
\item \emph{$\tilde{D}$ is bipartite.}\\
This follows from the fact that $\tilde{G}|\tilde{X}_{ij}=H|\tilde{X}_{ij}$
is bipartite. Let $(D_1,D_2)$ be the bipartition of $\tilde{D}$.
\item \emph{$|L(d)|=|L(d')|$ for every $d,d' \in D_1 \cup D_2$.}\\
Since $L(d) \subseteq \{1,2\}$ for every $d \in V(\tilde{D})$, and since
we have updated exhaustively, it follows that if $V(\tilde{D})$ meets
$X^0(L_Q)$, then $V(\tilde{D}) \subseteq X^0(L_Q)$.
\item \emph{There exists $a \in A$ with  a neighbor in $\tilde{D}$ and with
$L(a) \cap \{1,2\} \neq \emptyset$.}\\
This follows immediately from the definition of $D$.
\item \emph{Suppose that $|L(d)|=2$ for every $d \in V(\tilde{D})$.
We need to check that  for $\{i,j\}=\{1,2\}$, if $a \in A$ has a neighbor in $d \in D_1$ and 
$i \in L_Q(a)$, and $b \in B$ has a neighbor in $\tilde{D}$, then 
\begin{itemize}
\item if $b$ has a neighbor in $D_1$, then $j \not \in L_Q(b)$, and
\item if $b$ has a neighbor in $D_2$, then $i \not \in L_Q(b)$.
\end{itemize}}
\end{itemize}

We now check the condition of the last bullet.
Let $a \in A$ have a neighbor $d \in D_1$ and $1 \in L_Q(a)$. Suppose $b \in B$
has a neighbor in $D_1 \cup D_2$, and violates the conditions above. 
It follows from the definition of $Z^{12}$ and $B$ that $b \in \tilde{X}$
and  $|L_Q(b)|=2$. We may  assume that $b \in T_1(X)$. 
Since $|L_Q(b)|=2$, we deduce that  $L_Q(b)=L(b)=M_P(b)=L_P(T_1)$.
Since $b$ exists, $Q_1 \neq \emptyset$. 
Since $|L(d)|=2$  for every $d \in V(\tilde{D})$, it follows that $q_1$ is 
anticomplete to $D_1 \cup D_2$.  Since $b \not \in X^0(L_Q)$,
there is a component $D_0$ of $H|\tilde{X}_{12}$ such that
$q_1$ has a neighbor $d_0 \in V(D_0)$ and $b$ is anticomplete to $V(D_0)$.
Let  $\{i\}=L_Q(b) \cap \{1,2\}$, and let $\{1,2\} \setminus \{i\}=\{j\}$.
Then $j \not \in L_Q(b)=M_P(b)$, and so $j \not \subseteq L_P(T_1)$.
Consequently, there is $s \in S$ with $f(s)=j$, such that $s$ is complete
to $\tilde{X}(T_1)$. Since $V(\tilde{D}) \cup V(D_0) \subseteq X_{12}$,
it follows that $s$ is anticomplete to $V(\tilde{D}) \cup V(D_0)$.

Suppose first that $V(\tilde{D}) \neq \{d\}$. Since $b$ is not complete 
to $D_1 \cup D_2$ (because $L_Q(b) \cap \{1,2\} \neq \emptyset$), there is an 
edge $d_1d_2$ of $\tilde{D}$,  such that
$b$ is adjacent to $d_2$ and not to $d_1$. Now $d_1-d_2-b-s-q_1-d_0$
is a $P_6$ in $\tilde{G}_{12}(s)$, contrary to Lemma~\ref{P6free}.

This proves that $V(\tilde{D})=\{d\}$, and so $b$ is adjacent 
to $d$, $i=2$ and $j=1$. Therefore $L_P(T_1) \cap \{1,2\}=\{2\}$,
and so $L_Q(q_1)=c(q_1)=2$. Since $d_0 \in \tilde{X}_{12}$, it follows that
$L_Q(d_0)=1$. Since $1 \in L_Q(a)$ and $L_Q$ is obtained by exhaustive 
updating, we deduce that $a$ is non-adjacent to $d_0$. 
But now since $1 \in L_Q(a)$ and $f(s)=1$, we deduce that
$a-d-b-s-q_0-d_0$ is a path in $H_{12}(s)$ contradicting 
Lemma~\ref{keycompanion}. This proves that $(H,L_Q)$ has a $12$-insulating 
cutset with far side $Z^{12}$.

Let $\mathcal{L}=\{L_Q \; ; \; Q \in \mathcal{Q}\}$. Then 
$|\mathcal{L}| \leq |(V(G)|^{2m}$. Since $m \leq 2^{|S|}$, it follows that
$|\mathcal{L}| \leq |V(G)|^{2^{|S|}}$. Since $L_Q(v) \subseteq L(v)$ for 
every $v \in V(H)$, it follows that every  coloring of $(H,L')$ is
a coloring of $(H,L)$.

Now suppose that $(H,L)$ is colorable, and let $c$ be a coloring. 
We show that some $L' \in \mathcal{L}$ is colorable.
Let $i \in \{1, \ldots, m\}$. For a vertex $u \in \tilde{X}(T_i)$
define $val(u)=|G_{12}(u)|$. 
If some vertex $u$ of
$\tilde{X}(T_i)$ with a $12$-grandchild has $c(u) \in L(u) \cap \{1,2\}$, let $q_i$ be such a 
vertex with $val(q_i)$ maximum and set $Q_i=\{q_i\}$. If no such 
$u$ exists, let $Q_i=\emptyset$. 

Define $P_1, \ldots, P_m$ similarly replacing $\tilde{X}_{12}$ with 
$\tilde{X}_{34}$. Let 
$$Q=(Q_1, \ldots, Q_m, P_1, \ldots, P_m).$$ 
We show that   $c(v) \in L_Q(v)$ for every
$v \in V(H)$, and so $(H,L_Q)$ is colorable.
Since $L_Q$ is obtained from $L_Q'$ by updating, it is enough to prove that
$c(v) \in L_Q'(v)$. Suppose not. There are two possibilities (possibly
replacing $12$ with $34$).
\begin{enumerate}
\item $v \in \tilde{X}(T_i)$, $Q_i \neq \emptyset$, $G_{12}(q_i)$
is a proper subset of $G_{12}(v)$, and $c(v) \in \{1,2\}$;
\item $v \in \tilde{X}(T_i)$, $Q_i=\emptyset$, 
$G_{12}(v) \neq \emptyset$, and $c(v) \in \{1,2\}$.
\end{enumerate}
We show that in both cases we get a contradiction.
\begin{enumerate}
\item In this case $val(v) > v(q_i)$, contrary to the choice of $q_i$.
\item The existence of $v$ contradicts the fact that $Q_i=\emptyset$.
\end{enumerate}
This proves that $(H,L_Q)$ is colorable and completes the proof of
Theorem~\ref{insulatinglemma}. 
\end{proof}

Let $P=\esspc$ be an orthogonal excellent starred precoloring of a $P_6$-free 
graph $G$. We say that a near-companion triple $(H,L,h)$ is \emph{insulated} 
if for every $i \in \{2,3,4\}$ such that $Z^{1i}$ is non-empty,
$(H,L)$ has a  $1i$-insulating cutset $D \subseteq \tilde{X}$ with far side $Z^{1i}$. 
We can now prove the main result of this section.

\begin{theorem}
\label{insulatinglist} 
There is function $q: \mathbb{N} \rightarrow \mathbb{N}$ such that the
following holds.  Let $G$ be a $P_6$-free graph, let $P=\esspc$ be an
orthogonal excellent starred precoloring of $G$, and let $(H, L, h)$
be a near-companion triple for $P$. There is an algorithm with running
time $O(|V(G)|^{q(|S|)})$ that outputs a collection $\mathcal{L}$ of
$4$-list assignments for $H$ such that
\begin{itemize}
\item $|\mathcal{L}| \leq |V(G)|^{q(|S|)}$.
\item If $L' \in \mathcal{L}$  and $c$ is a proper coloring of $(H,L')$, then
$c$ is a proper coloring of $(H,L)$.
\item If $(H,L)$ is colorable, there exists $L' \in \mathcal{L}$ such that
$(H,L')$ is colorable.
\end{itemize}
Moreover, for every  $L' \in \mathcal{L}$.
\begin{itemize}
\item $L'(v) \subseteq L(v)$ for every $v \in V(H)$.
\item $(H,L',h)$ is insulated.
\end{itemize}
\end{theorem}

\begin{proof}
Let $\mathcal{L}_2$ be as in Lemma~\ref{insulatinglemma}.
By symmetry, we can apply Lemma~\ref{insulatinglemma} with $12$ replaced by
$13$ to $(H,L',h)$ for every $L' \in \mathcal{L}_2$; let $\mathcal{L}_3$
be the union of all the collections of lists thus obtained.
Again by symmetry, we can apply Lemma~\ref{insulatinglemma} with $12$ replaced 
by $14$ to $(H,L',h)$ for every $L' \in \mathcal{L}_3$; let $\mathcal{L}_4$
be the union of all the collections of lists thus obtained. Now
$\mathcal{L}_4$ is the required collection of lists.
\end{proof}

\section{Divide and Conquer}\label{sec:2sat}

The main result of this section is the last piece of machinery that we need to
solve the 4-precoloring-extension problem.

We need the following two facts.
\begin{theorem} \cite{edwards}
\label{Edwards}
There is a polynomial time algorithm that  tests, for  graph $H$ and a list 
assignment $L$ with $|L(v)| \leq 2$ for every $v \in V(H)$, if
$(H,L)$ is colorable, and finds a proper coloring if one exists.
\end{theorem}

\begin{theorem}\cite{2SAT}
\label{2SATthm}
The 2-SAT problem can be solved in polynomial time.
\end{theorem}

We prove:

\begin{lemma}
\label{farside}
Let $G$ be a $P_6$-free graph and let $P=\esspc$ be an orthogonal excellent starred  precoloring of $G$. Let $(H,L',h)$ be a companion triple for
$P$, where $V(H)=\tilde{X} \cup Z$, as in the definition of a companion 
triple.  Assume that $D \subseteq \tilde{X}$ 
is a $12$-insulating chromatic cutset in $(H,L')$ with far side $Z^{12}$.
There is a polynomial time algorithm that test if $(H|(Z^{12} \cup D), L')$
is colorable, and finds a proper coloring if one exists.
\end{lemma}

\begin{proof}
We may assume that $X_0=X^0(P)$.
Let $\tilde{P}=(\tilde{G},S,X_0,\tilde{X},Y^*, f)$ be as in the 
definition of a companion triple, where $V(H)=\tilde{X} \cup Z$.
By Theorem~\ref{Edwards} we can test in polynomial time
if $H|(D \cap \tilde{X}_{12},L')$ and  $H|(D \cap \tilde{X}_{34},L')$ is 
colorable.
If one of these pairs is not colorable,  stop and output  that 
$(H|(Z^{12} \cup D), L)$ is not colorable. So we may assume
both the pairs are colorable, and in particular every component
of $H|(D \cap \tilde{X}_{12})$ and $H|(D \cap \tilde{X}_{34})$ is bipartite.

We modify $L'$ without changing the colorability property. 
First, let $L''$ be obtained from $L'$ by  updating exhaustively from 
$X^0(L')$.
Next if $v \in V(H) \setminus \tilde{X}_{12}$ has a neighbor on both sides of the bipartition of a component of $H|\tilde{X}_{12}$, we remove both $1$ and $2$ from $L''(v)$, and the same for $\tilde{X}_{34}$; call the resulting list
assignment $L$. (We have already done a similar modification while constructing
list assignments $L_Q$ in the proof of Lemma~\ref{insulatinglemma}, but there 
we only  modified lists of vertices in $\tilde{X}$, so this step is not 
redundant.) Set $f(u)=L(u)$ for every $u \in X^0(L)$.  Clearly:

 \vspace*{-0.4cm}
  \begin{equation}\vspace*{-0.4cm} \label{updated}
    \longbox{\emph{
If $v \in V(H)$
is adjacent to $x \in X^0(L)$, then $L(v) \cap L(x)=\emptyset$.}}
\end{equation}

Next we prove:

 \vspace*{-0.4cm}
  \begin{equation}\vspace*{-0.4cm} \label{X0nbr}
    \longbox{\emph{
Let  $\{p,q\} \in \{\{1,2\},  \{3,4\}\}$ and let
$z \in Z^{12}$ with $|L(z) \cap \{p,q\}|=1$. Let
$L(z) \cap \{p,q\}=\{i\}$ and $\{p,q\} \setminus L(z)=\{j\}$.
Then there exists $y \in V(h(z))$ and $u \in S \cup X_0 \cup X^0(L)$ such that
$f(u)=j$ and $uy \in E(\tilde{G})$.}}
\end{equation}

To prove (\ref{X0nbr}) let $z \in Z$ with $L(z) \cap \{1,2\}=\{1\}$
(the other cases are symmetric). Since $1 \in L(z)$, it follows
that $z$ does not have neighbors on both sides of the bipartition of a 
component of $H|\tilde{X}_{12}$, and therefore $L(z)=L''(z)$. If
$2 \not \in L'(z)$, then such $u$ exists from the definition of
a near-companion triple, so we may assume $2 \in L'(z)$. This implies 
that there is $u \in X^0(L)$ such that $u$ is adjacent to $z$, and $f(u)=2$.
Since $Z$ is stable, it follows that $u \in \tilde{X} \cup X_0 \cup S$,
and so  $u$ is complete to $V(h(z))$, and
(\ref{X0nbr}) follows. 

\bigskip

We define an instance $I$ of the $2$-SAT problem.
The variables are the vertices of $Z^{12}$, and the clauses are as follows:
\begin{enumerate}
\item For every $z_1, z_2 \in Z^{12}$, if $L(z_i) \cap \{1,2\}=\{i\}$ for
$i=1,2$ and $z_1, z_2$ have neighbors on the same side of the bipartition of
some component of $H|(D \cap \tilde{X}_{12})$, add the clause 
$(\neg{z_1} \vee \neg{z_2})$.
\item For every $z_1, z_2 \in Z^{12}$, if 
$L(z_1) \cap \{1,2\}=L(z_2) \cap \{1,2\} \in \{\{1\},\{2\}\}$ for
$i=1,2$ and $z_1, z_2$ have neighbors on opposite sides of the bipartition of
some component of $H|(D \cap \tilde{X}_{12})$, 
add the clause  $(\neg{z_1} \vee \neg{z_2})$.
\item For every $z_1, z_2 \in Z^{12}$, if $z_1,z_2$  have neighbors on the same
  side of the
  bipartition of
  some component of $H|(D \cap \tilde{X}_{12})$, and also $z_1,z_2$ have
  neighbors on opposite sides of the bipartition of
some component of $H|(D \cap \tilde{X}_{12})$, 
  add the clause 
$(\neg{z_1} \vee \neg{z_2})$.
\item For every $z_3, z_4 \in Z^{12}$, if $L(z_i) \cap \{3,4\}=\{i\}$ for
$i=3,4$ and $z_3, z_4$ have neighbors on the same side of the bipartition of
some component of $H|(D \cap \tilde{X}_{34})$,
add the clause $(z_3 \vee z_4)$. 
\item For every $z_3, z_4 \in Z^{12}$, if 
$L(z_3) \cap \{3,4\}=L(z_4) \cap \{3,4\} \in \{\{3\},\{4\}\}$ for
$i=3,4$ and $z_3, z_4$ have neighbors on opposite sides of the bipartition of
some component of $H|(D \cap \tilde{X}_{34})$, 
add the clause  $(z_3 \vee z_4)$.
\item For every $z_3, z_4 \in Z^{12}$, if 
$z_3, z_4$ have neighbors on the same side of the bipartition of
  some component of $H|(D \cap \tilde{X}_{34})$, and also $z_3,z_4$  have
  neighbors on opposite sides of the bipartition of
some component of $H|(D \cap \tilde{X}_{34})$, 
add the clause $(z_3 \vee z_4)$. 
\item If $z \in Z^{12}$ and $L(z) \subseteq \{1,2\}$, add the clause
$(z \vee z)$.
\item If $z \in Z^{12}$ and $L(z) \subseteq \{3,4\}$, add the clause 
$(\neg{z} \vee \neg{z})$.
\end{enumerate}

By Theorem~\ref{2SATthm} we can test in polynomial time if $I$ is satisfiable.

We claim that $I$ is satisfiable if and only if $(H|(Z^{12} \cup D),L)$
is colorable, and a proper coloring of $(H|(Z^{12} \cup D),L)$ can be
constructed in polynomial time from a satisfying assignment for $I$.

Suppose first that $(H|(Z^{12} \cup D),L)$ is colorable, and let
$c$ be a proper coloring. 
For $z \in Z^{12}$, set $z=TRUE$ if $c(z) \in \{1,2\}$ and $z=FALSE$
if $c(z) \in \{3,4\}$. It is easy to check that every clause is satisfied.

Now suppose that $I$ is satisfiable, and let $g$ be a satisfying assignment.
Let $A'$ be the set of vertices $z \in Z^{12}$ with $g(z)= TRUE$, and let
$B'=Z^{12} \setminus A'$.
Let $A=A' \cup (D \cap \tilde{X}_{12})$ and $B=B' \cup (D \cap \tilde{X}_{34})$.
For $v \in A$ let $L_A(v)=L'(v) \cap \{1,2\}$, and for
$v \in B$ let $L_B(v)=L'(v) \cap \{3,4\}$. In order to show that
$(H|(Z^{12} \cup D),L)$ is colorable and find a proper coloring, it is enough 
to  prove that
$(H|A, L_A)$ and $(H|B, L_B)$ are colorable, and find their proper colorings. 
We show that $(H|A,L_A)$ is colorable; the argument for $(H|B,L_B)$  is 
symmetric.

Since for every $z \in Z^{12}$ with $L(z) \subseteq \{3,4\}$ 
$(\neg{z} \vee \neg{z})$ is a clause  (of type 8) in $I$, it
follows that $L(z) \cap \{1,2\} \neq \emptyset$ for every $z \in A$.
Let $A_1=\{v \in A \; : \; L_A(v)=\{1\}\}$, 
$A_2=\{v \in A \; : \; L_A(v)=\{2\})$, and  $A_3=A \setminus (A_1 \cup A_2)$
Let $F$ be a graph defined as follows.
$V(F)=(A_3 \cup \{a_1,a_2\})$, where
$F \setminus \{a_1,a_2\}=H|A_3$, $a_1a_2 \in E(F)$, 
and for $i=1,2$ $v  \in A_3$ is adjacent to 
$a_i$ if and only if $v$ has a neighbor in $A_i$ in $H$.

We claim that
$(H|A, L_A)$ is colorable if and only if $F$ is bipartite;
and if $F$ is bipartite, then a proper coloring of $(H|A,L_A)$ can be 
constructed in polynomial time. 
Suppose  $F$ is bipartite and let $(F_1,F_2)$ be the bipartition.
We may assume $a_i \in F_i$. Let $i \in \{1,2\}$. 
For every 
$v \in (F_i \cup A_i) \setminus \{a_i\}$, we have that $i \in L_A(v)$, and
so we can set $c(v)=i$. This proves that $(H|A, L_A)$ is colorable,
and constructs a proper coloring.
Next assume that $(H|A, L_A)$ is colorable. For $i=1,2$, let  $F_i'$ be the 
set of vertices of $A$ colored $i$. Then $A_i \subseteq F_i'$, and
setting $F_i=(F_i' \setminus A_i) \cup \{a_i\}$, we get that
$(F_1,F_2)$ is a bipartition of $F$. This proves the claim.

Finally we show that $F$ is bipartite. Recall that the pair  
$(H|(D \cap \tilde{X}_{12}),L)$ is colorable, and therefore
$H|(D \cap \tilde{X}_{12})$ is bipartite.
Since $L_A(v) \subseteq L(v)$ for every $v \in A_3$, 
and $L_A(v) \cap \{1,2\} \neq \emptyset$ for every $v \in A$, it follows
that no vertex of $A \cap Z^{12}$ has a neighbor on two opposite sides of a
bipartition of a component of $H|(D \cap \tilde{X}_{12})$.
First we show that $H|A$ is bipartite. Suppose that there is an odd induced
cycle $C$ in $H|A$. It follows that $|V(C)|=5$ and, since
$Z^{12}$ is stable, $|C \cap Z^{12}|=2$.
But then some clause of type 3 or 6 is not satisfied, a contradiction.
This proves that $H|A$ is bipartite. 

Suppose that $F$ is not bipartite. Then there is an odd cycle $C$ in $F$,
and so $V(C) \cap \{a_1,a_2\} \neq \emptyset$. In $H$ this implies  that
there is a path $T=t_1-\ldots-t_k$  with 
$\{t_2, \ldots, t_{k-1}\} \subseteq A_3$, such that either 
\begin{itemize}
\item $k$ is even, and for some  $i \in \{1,2\}$ $t_1, t_k \in A_i$, or  
\item $k$ is odd, $t_1 \in A_1$, and $t_k \in A_2$.
\end{itemize}
Since $T$ is a path in $H|(Z \cup \tilde{X}_{12})$, it follows that  $k \leq 5$.
If $t_1 \in \tilde{X}_{12} \cap D$, then $t_1 \in X^0(L)$, and so  
by (\ref{updated}),   $t_2 \in A_1 \cup A_2$, a contradiction. 
This proves that $t_1 \in Z^{12}$, and similarly $t_k \in Z^{12}$.

Suppose first that $k$ is even. Since  $Z^{12}$ is stable, it follows that 
$k \neq 2$, and so $k=4$.  Since $t_1, t_4 \in Z^{12}$ and 
since $Z^{12}$ is stable, it follows that $t_2, t_3 \in \tilde{X}_{12}$.
But now $(\neg{t_1} \vee \neg{t_4})$ is a clause (of type 2) in $I$, and
yet $g(t_1)=g(t_4)=TRUE$, a contradiction.

This proves that $k$ is odd. If $k=3$ then, since $Z^{12}$ is stable, 
$t_2  \in \tilde{X}_{12}$, and so $(\neg{t_1} \vee \neg{t_3})$ is a 
clause (of type 1) in $I$, and yet  $g(t_1)=g(t_3)=TRUE$, a contradiction. 
This proves that $k=5$. Since $Z^{12}$ is 
stable, it follows that $t_2,t_4 \in \tilde{X}_{12}$. If 
$t_3 \in \tilde{X}_{12}$, then $(\neg{t_1} \vee \neg{t_5})$ is a
clause (of type 1) in $I$, contrary to the fact that both
$g(t_1)=g(t_5)=TRUE$, a contradiction. Therefore $t_3 \in Z^{12}$.
We may assume that $t_1 \in A_1$. By (\ref{X0nbr}) there exist
$u \in S \cup X_0 \cup X^0(L)$ and $y_1 \in V(h(t_1))$ such that
$f(u)=2$ and $uy_1 \in E(\tilde{G})$. Since $t_2 \in \tilde{X}$, 
it follows that $t_2$ is complete to $V(h(t_1))$, and in particular $t_2$ is
adjacent to $y_1$. Since $X_0=X^0(P)$, it follows that $u$ is anticomplete
to $\{t_2,t_4\}$. Let $i \in \{3,5\}$. By the definition of a companion triple, 
since    $2 \in L(t_i)$, there exists $y_i \in V(h(t_i))$
such that $u$ is non-adjacent to $y_i$ in $\tilde{G}$. Now since no vertex of 
$\tilde{X}$
is mixed on a component to $\tilde{G}|Y^*$, it follows that
$u-y_1-t_2-y_3-t_4-y_5$  is a $P_6$ in $\tilde{G}_{12}(u)$, contrary to 
Lemma~\ref{P6free}.   This proves Lemma~\ref{farside}. 
\end{proof}

\section{The complete algorithm} \label{sec:complete}

First we prove Theorem~\ref{excellent}, which we restate.

\begin{theorem}
\label{excellent2}
For every integer $C$ there exists a polynomial-time algorithm with the following specifications.
\\
\\
{\bf Input:}  An excellent starred precoloring $P=\esspc$ of a 
$P_6$-free graph $G$ with $|S| \leq C$.
\\
\\
{\bf Output:} A precoloring extension of $P$ or a determination
that none exists.
\end{theorem}

\begin{proof}
By Theorem~\ref{orthogonalthm} 
we can construct in polynomial time a collection
$\mathcal{L}$ of orthogonal excellent starred   precolorings of $G$, such that in order to determine if $P$ has a precoloring extension (and find one if it 
exists), it is enough to check if each element of $\mathcal{L}$ has a 
precoloring extension, and find one if it exists.
Thus let $P_1 \in \mathcal{L}$. 
By Theorem~\ref{companion} we can construct in polynomial time a
companion triple $(H,L,h)$ for $P_1$, and it is enough to check if
$(H,L,h)$ is colorable.

Now proceed as follows. If $L(v)=\emptyset$ for some
$v \in V(H)$, stop and output ``no precoloring extension''.
So we may assume $L(v) \neq \emptyset$ for every $v \in V(H)$.
Let $\mathcal{L}$ be a collection of
lists as in Theorem~\ref{insulatinglist}.
If $\mathcal{L}=\emptyset$, stop and output ``no precoloring extension'',
so we may assume that $\mathcal{L} \neq \emptyset$.
Let $L' \in \mathcal{L}$; then $(H,L',h)$ is insulated.
For every $i$ let $D^i$ be and insulating $1i$-cutset with far side $Z^{1i}$, 
and  let ${D^i}'=\{d \in D_i \; : \; |L'(d)|=2\}$. 
Let $H_i=H|(D^i \cup Z^{1i})$, and let 
$H_1=H \setminus \bigcup_{i=2}^4({D^i}'\cup Z^{1i})$.
Observe that $V(H_1) \subseteq \tilde{X}$.
By Lemma~\ref{farside}, we can check if each of the pairs
$(H_i,L')$  with $i \in \{2,3,4\}$ is colorable, and by Theorem~\ref{Edwards}, 
we can check if $(H_1,L')$ 
is colorable and find a proper coloring if one exists.
If one of these pairs  is not colorable,
 stop and output ``no precoloring extension''.
So we may assume that $(H_i,L')$ is colorable  for every 
$i \in \{1, \ldots, 4\}$. 
Observe that $D^2$ is an insulating
$12$-cutset in $(H|(V(H_1) \cup V(H_2)),L')$ with far side $Z^{12}$, 
$D^3$ is an insulating
$13$-cutset in $(H|(V(H_1) \cup V(H_2) \cup V(H_3)),L')$ with far side $Z^{13}$, 
and $D^4$ is an insulating $14$-cutset in $(H,L')$ with far side $Z^{14}$.
Now three applications of Theorem~\ref{insulating} show that
$(H,L)$ is colorable, and produce a proper coloring.
This proves~\ref{excellent2}.
\end{proof}

We can now prove the main result of the series, the following.
\begin{theorem}
\label{complete}
There exists a polynomial-time algorithm with the following specifications.
\\
\\
{\bf Input:}  A 4-precoloring $(G,X_0,f)$ of a $P_6$-free graph $G$.
\\
\\
{\bf Output:} A precoloring extension of $(G,X_0,f)$ or a determination
that none exists.
\end{theorem}

\begin{proof}
Let  $\mathcal{L}$ be as in Theorem~\ref{Yaxioms}. Then $\mathcal{L}$
can be constructed in polynomial time, and it is enough to check if
each element of $\mathcal{L}$ has a precoloring extension, and find one if it 
exists. Now apply the algorithm of Theorem~\ref{excellent2} to every element 
of $\mathcal{L}$.
\end{proof}

\section{Acknowledgments}
This material is based upon work supported in part by the U. S. Army  Research 
Laboratory and the U. S. Army Research Office under    grant number 
W911NF-16-1-0404. The authors are also grateful to Pierre Charbit, Bernard Ries,
Paul Seymour, Juraj Stacho and Maya Stein for many useful discussions.


\begin{thebibliography}{50}
\bibitem{c1} Bonomo, Flavia, Maria Chudnovsky, Peter Maceli, Oliver Schaudt, Maya Stein and Mingxian Zhong. \emph{Three-coloring and list three-coloring graphs without induced paths on seven vertices.} Appeared on-line in Combinatorica (2017), DOI:10.1007/s00493-017-3553-8.

\bibitem{ICM}  Chudnovsky, Maria. 
\emph{Coloring graphs with forbidden induced subgraphs.}
Proceedings of the ICM (2014): 291--302.


\bibitem{Paper2} Chudnovsky, Maria, Sophie Spirkl, and Mingxian Zhong.
\emph{Four-coloring $P_6$-free graphs. II. Finding an excellent precoloring.} arXiv:1802.02283. 

\bibitem{edwards} Edwards, Keith. \emph{The complexity of colouring problems on dense graphs.} Theoretical Computer Science 43 (1986): 337--343.

\bibitem{hoang} Ho\`ang, Ch\'inh T., Marcin Kami\'nski, Vadim Lozin, Joe Sawada, and Xiao Shu. \emph{Deciding $k$-colorability of $P_5$-free graphs in polynomial time.} Algorithmica 57, no. 1 (2010): 74--81.

\bibitem{huang} Huang, Shenwei. \emph{Improved complexity results on $k$-coloring $P_t$-free graphs.} European Journal of Combinatorics 51 (2016): 336--346.

\bibitem{gps} Golovach, Petr A., Dani\"el Paulusma, and Jian Song. \emph{Closing complexity gaps for coloring problems on H-free graphs.} Information and Computation 237 (2014): 204--214.



\bibitem{2SAT} Krom, Melven R. \emph{The decision problem for a class of
first order formulas in which all disjunctions are binary.}
Zeitschrift der Mathematik 13 (1967): 15--20.


\end{thebibliography}
\end{document}